\documentclass[letter,12pt]{article}

\usepackage{amssymb}
\usepackage{amsthm}
\usepackage{amsfonts}
\usepackage[latin2]{inputenc}
\usepackage{amsmath}
\usepackage{anysize}
\usepackage{hyperref}
\usepackage{bbm}
\marginsize{2.6cm}{2.6cm}{1cm}{2cm}

\newtheorem*{theorem*}{Theorem}
\newtheorem{theorem}{Theorem}[section]

\newtheorem{lemma}[theorem]{Lemma}
\newtheorem{corollary}[theorem]{Corollary}
\newtheorem{remark}[theorem]{Remark}

\newtheorem{definition}[theorem]{Definition}
\newtheorem{proposition}[theorem]{Proposition}
\newtheorem{conjecture}[theorem]{Conjecture}

\def\MA{Monge--Amp\`ere }

\def\i{\sqrt{-1}}
\def\del{\partial}
\def\dbar{\bar\partial}
\def\ddbar{\del\dbar}
\def\ra{\rightarrow}

\newcommand{\RR}{\mathbb{R}}

\newcommand{\Cc}{\mathcal{C}}
\def\del{\partial}

\newcommand{\psh}{{\rm PSH}}

\DeclareMathOperator{\Ric}{Ric}

\def\w{\wedge}
\def\o{\omega}

\def\Ent{{\rm Ent}}
\def\Ec{\mathcal{E}}
\def\vep{\varepsilon}
\def\AM{{\rm AM}}
\def\Kc{\mathcal{K}}
\def\Kcc{\mathcal{K}_{\chi}}

\title{Convexity of the extended K-energy and the large time behaviour of the weak Calabi flow \vspace{-0.1in}}

\author{Robert J. Berman, Tam\'as Darvas, Chinh H. Lu\thanks{The first and the third author are supported by the European Research Council. The third author is on leave from Chalmers University of Technology. The second author is supported by BSF grant 2012236.}}
\date{\vspace{-0.4in}}

\begin{document}
\maketitle

\begin{abstract}
Let $(X,\o)$ be a compact connected K\"ahler manifold and denote by $(\mathcal E^p,d_p)$ the metric completion of the space of K\"ahler potentials $\mathcal H_\o$ with respect to the $L^p$-type path length metric $d_p$. First, we show that the natural analytic extension of the (twisted) Mabuchi K-energy to $\mathcal E^p$ is a $d_p$-lsc functional that is convex along finite energy geodesics. Second, following the program of J. Streets, we use this to study the asymptotics of the weak (twisted) Calabi flow inside the  CAT(0) metric space $(\mathcal E^2,d_2)$. This flow exists for all times and coincides with the usual smooth (twisted) Calabi flow whenever the latter exists. 
We show that the weak (twisted) Calabi flow either diverges with respect to the $d_2$-metric or it $d_1$-converges to some minimizer of the K-energy inside $\mathcal E^2$. This  gives the first concrete result about the long time convergence of this flow on general K\"ahler manifolds, partially confirming a conjecture of Donaldson. 
Finally, we investigate the possibility of constructing destabilizing geodesic rays asymptotic to diverging weak (twisted) Calabi trajectories, and give a result in the case when the twisting form is K\"ahler. If the twisting form is only smooth, we reduce this problem to a conjecture on the regularity of minimizers of the K-energy on $\mathcal E^1$, known to hold in case of Fano manifolds.
\end{abstract}


\section{Introduction}

Given a compact connected K\"ahler manifold $(X,\o)$, we denote by $\mathcal H$ the space of smooth K\"ahler metrics in the cohomology class $[\o]$. As follows from the $\ddbar$-lemma of Hodge theory, up to a constant, this space is in a one-to-one correspondence with the space of K\"ahler potentials:
$$\mathcal H_\o = \{ u \in C^\infty(X) : \ \o_u:=\o + i\partial \bar \partial u >0\}.$$
As $\mathcal H_\o$ is an open subset of $C^\infty(X)$, it is a Fr\'echet manifold and it is possible to endow it with different $L^p$-type Finsler metrics $p \geq 1$:
\begin{equation}\label{FinslerDef}
\|\xi\|_{p,u}:=  \bigg( V^{-1}\int_X |\xi|^p \o_u^n \bigg)^{1/p}, \  \xi \in T_u \mathcal{H}_\o=C^\infty(X).
\end{equation}
For $p=2$ one recovers the Riemannian structure of Mabuchi which turns  $\mathcal H$ into a Riemannian symmetric space of constant negative curvature \cite{mab,do1,se}, 
but as will be explained below, the Finsler case $p=1$ will also play a key role in the present paper.

One of the central questions of K\"ahler geometry, going back to Calabi, is to understand under what conditions $\mathcal H$ contains a constant scalar curvature K\"ahler (csc-K) metric. From a variational point of view this amounts to looking for critical points (minimizers) of Mabuchi's K-energy functional $\mathcal K :\mathcal H_\o \to \Bbb R$ \cite{mab,do1}: whose first variation is defined by the following formula
$$\langle D\mathcal K(u), \delta u\rangle = V^{-1} \int_X \delta u (\bar S-S_{\o_u})\o^n_u,$$
where $V = \int_X \o^n$ is the total volume and $\bar S=nV^{-1}\int_M \Ric \o \wedge \o^{n-1}=V^{-1}\int_M S_\o  \o^{n}$ is the mean scalar curvature.  According to a formula of Chen and Tian the K-energy can be expressed explicitly in terms of the K\"ahler potential: 
\begin{equation}
\label{KEnergyDef}
\mathcal K(u):= 
\textup{Ent}(\o^n,\o^n_u)
+  \bar{S}\textup{AM}(u) 
-  n\textup{AM}_{\Ric \o}(u),
\end{equation}
where $\textup{Ent}(\o^n,\o^n_u)=V^{-1}\int_X \log(\o_u^n/\o^n)\o_u^n$ is the entropy of the measure $\o_u^n$ with respect to $\o^n$ and $\textup{AM},\textup{AM}_\gamma: \mathcal H_{\omega} \to \Bbb R$ is the  Aubin-Mabuchi (also Aubin-Yau) energy and its ``$\gamma$-contracted" version:
$$\textup{AM}(u)= \frac{1}{(n+1)V} \sum_{j=0}^{n} \int_X u \o_u^{j} \wedge \o^{n-j}, \ \ \ \textup{AM}_{\gamma}(u) = \frac{1}{nV}\sum_{j=0}^{n-1}\int_X u {\gamma} \wedge \o_u^{j} \wedge \o^{n-1-j}.$$
As shown by Mabuchi, the K-energy is convex along geodesics in $\mathcal H_\o$, when the geodesics are defined in terms of the corresponding $L^2$-Riemannian structure. However a major technical stumbling block in this infinite dimensional setting is that the Riemmanian
structure on $\mathcal H_\o$ is not geodesically complete and this is one of the reasons that we will be forced to work with various completions of  $\mathcal H_\o$, as discussed below. 

In the finite dimensional Riemannian setting a time-honoured approach of finding minimizers of convex functions is to follow their negative (downward) gradient flow. In the present infinite dimensional Riemannian setting the negative gradient flow of the K-energy is precisely the Calabi flow
  $t \to c_t$:
\begin{equation*}\label{CalabiFlowDef}
\frac{d}{dt} c_t = S_{\o_{c_t}} - \bar S.
\end{equation*}
Given arbitrary initial potential $c_0 \in \mathcal H_\o$, short time existence of the flow is due to Chen-He \cite{ch1}, but long time existence is still an open conjecture due to Calabi-Chen. In case $\dim X=1$, long time existence and convergence of the flow was first explored by Chruciel \cite{ch}. Fine used finite dimensional flows to approximate the Calabi flow \cite{fi}. Under various restrictive conditions, convergence and existence theorems for the Calabi flow have been extensively studied. We refer the reader to \cite{ch2, fh,he,hu1,hz,lwz,sz,tw} to cite a few works from a very fast growing literature.

The main motivation of our paper is the following conjecture of Donaldson  on the long time asymptotics and convergence of the Calabi flow, which roughly stated says the following:

\begin{conjecture} \textup{\cite{do2}} \label{DonaldsonConj} Let $[0,\infty) \ni t \to c_t \in \mathcal H$ be a Calabi flow trajectory. Exactly one of the following alternatives hold:
\vspace{-0.1in}
\begin{itemize}
\setlength{\itemsep}{1pt}
    \setlength{\parskip}{1pt}
    \setlength{\parsep}{1pt}   
\item[(i)]The curve $t \to c_t$ converges smoothly to some csc-K potential $c_\infty \in \mathcal H_\o$ as $t \to \infty$.
\item[(ii)] The curve $t \to c_t$ diverges as $t \to \infty$ and encodes destabilizing information about the K\"ahler structure.
\end{itemize} 
\end{conjecture}

We refer to \cite{do2} for a precise statement and further details about this conjecture. To avoid the difficulties arising in PDE theory related to long time existence, we recast the Calabi flow in the metric completion of ${(\mathcal H_\o,d_2)}$ following Streets \cite{st1,st2}, who applied the work of Mayer \cite{may} and Ba\u c\'ak \cite{ba} concerning gradient flows of convex functionals on Hadamard spaces (i.e. CAT(0)-spaces) to the setting of the ``minimizing movement" Calabi flow. Before we can do this however, we need to understand how the K-energy extends to certain spaces of singular potentials.
The key new feature of our approach is that we take advantage of the fact that the corresponding abstract metric space (defined in terms of Cauchy sequences in \cite{st1}) can be realized concretely in terms of certain singular K\"ahler potentials, i.e. using pluripotential theory, which in particular allows us to improve on the abstract convergence result in \cite{st2}.  

\paragraph{Finite energy spaces and extensions of the twisted K-energy.} 

In order to briefly introduce our setting, we denote by $\mathcal{E}^{p}$ the space of $\omega$-psh functions on $X$ which have finite energy with respect to the standard $p-$homogenous weight, as introduced
by Guedj-Zeriahi \cite{gz1}. As shown in \cite{da2}, the abstract
metric completion of the $L^{p}-$type Finsler metric on $\mathcal{H_\o}$ \eqref{FinslerDef} may be identified with the finite energy space $\mathcal{E}^{p}$ equipped with a natural distance function that we will denote by $d_{p},$ which is comparable to an explicit energy type expression  \eqref{dpCharFormula}. When $p=2$, this identification was conjectured by Guedj in
\cite{gu}. Furthermore, in the case $p=1$, it yields a Finsler
realization $(\mathcal{E}^{1},d_{1})$ of the strong topology on $\mathcal{E}^{1}$ introduced in \cite{bbegz} (which can be seen as a higher dimensional ``non-linear" generalization of the classical strong topology defined by the Dirichlet norm on a Riemann surface). 
   
Moreover, as shown in \cite{da1,da2}, for any pair of potentials $u_0,u_1 \in \mathcal E^p$ one can construct a $d_p$-geodesic segment (in the metric sense) 
explicitly, as a decreasing pointwise limit of $C^{1,\bar 1}$-weak geodesics (in the sense of Chen $\cite{c}$, i.e. as $C^{1,\bar 1}$-solutions to certain complex Monge-Amp\`ere equations). These $d_p$-geodesic segments will be referred to as \emph{finite energy geodesics} in the future and we direct the reader to Section \ref{subsect: the space Ep} for more details. A recurrent theme in the present work is the interaction between the cases $p=2$ and  $p=1$, which in particular will allow us to exploit the energy/entropy compactness theorem from  \cite{bbegz} to get a convergence result for the Calabi flow with respect to the $d_1$-topology. This strengthens the general convergence result 
of \cite{st2}, concerning the weak $d_2$-topology, which does not imply any convergence in the sense of pluripotential theory (Remark \ref{rem: weakd2convergence}).

Our starting point is the observation that the K-energy functional $\mathcal K$ originally defined on $\mathcal H_\o$ admits a natural ``analytic extension" to the finite energy space  $\mathcal E^1$  (and hence by restriction to all spaces  $\mathcal E^p$). This is simply the extension obtained by interpreting the entropy part (the first term) and the energy part (the second two terms) in formula \eqref{KEnergyDef} in the general sense of  probability theory and pluripotential theory, respectively; essentially as in the Fano setting previously considered in \cite{brm1,bbegz}. As we will see, the energy part is $d_1$-continuous whereas the entropy part is only $d_1$-lsc, and in the particular  case of $C^{1,\bar 1}$-potentials, this extension coincides with the one introduced by Chen \cite{c5}. We then go on to show that the restriction to  $\mathcal E^p$  of the analytic extension coincides with the canonical "topological extension" of the K-energy, i.e. the greatest $d_p$-lsc extension from $\mathcal H_\o$. In particular, applied to the case $p=2$, which is the one relevant to the Calabi flow, this yields an analytic formula for Streets' extension of the K-energy.

The analytic extension formula allows us to establish the convexity of the extended K-energy along finite energy geodesics, using an approximation argument and the $C^{1,\bar 1}$-case recently settled in \cite[Theorem 1.1]{bb} (originally conjectured by Chen).  

Before we state our first theorem, recall that in various applications of K\"ahler geometry it is necessary to deal with the more general concept of twisted csc-K metrics and the corresponding twisted-K energy (e.g. \cite{c4, cpz,de,fi1,sto}). 
As it takes little extra effort, throughout this paper we work at this level of generality, with $\chi$ denoting a very general  twisting form \eqref{eq: ChiProp} and  $\mathcal K_\chi$ the corresponding twisted K-energy \eqref{KEnTwistAltDef}. The relevant terminology will be recalled in Section \ref{sec: Prelim the twisted K-energy}.

\begin{theorem}[Theorem \ref{ExtKEnergyThm}] \label{ExtKEnergyThmIntr} Suppose $(X,\o)$ is a compact connected K\"ahler manifold. The K-energy can be extended to a functional $\mathcal K:\mathcal E^1 \to (-\infty, \infty]$ using formula  \eqref{KEnergyDef}. 
The restricted functional $\mathcal K|_{\mathcal E^p}$ is the greatest $d_p$-lsc extension of $\mathcal K|_{\mathcal H_\o}$ for any $p\geq 1$. Additionally, $\mathcal K|_{\mathcal E^p}$ is convex along the finite energy geodesics of $\mathcal E^p$. If $\chi =\beta + i\ddbar f$ satisfies \eqref{eq: ChiProp}, the corresponding result also holds for the twisted K-energy $\mathcal K_\chi$.
\end{theorem}

An important ingredient in the proof of Theorem \ref{ExtKEnergyThmIntr} is understanding  approximation of potentials of $\mathcal E^p$ while also approximating entropy. In this direction, we note the following theorem. More precise results can be obtained using the flow techniques of \cite{gz2,dinl}, and will be discussed elsewhere. 

\begin{theorem}[Theorem \ref{thm: Ep approximation with entropy}] Suppose $u \in \mathcal E^p$ and $f$ is a usc function on $X$ satisfying $e^{-f}\in L^1(X,\omega^n)$. Then one can find $u_k \in \mathcal H_\o$ with $d_p(u_k,u) \to 0$  and $\textup{Ent}(e^{-f}\o^n,\o_{u_k}^n) \to \textup{Ent}(e^{-f}\o^n,\o_{u}^n).$
\end{theorem}

Finally, as a consequence of Theorem \ref{ExtKEnergyThmIntr} we obtain that the space of finite $\chi$-entropy potentials $\textup{Ent}_\chi(X,\o)$ is geodesically closed, and in case $\Ric \o \geq \beta$ the twisted entropy is convex along finite energy geodesics, giving the K\"ahler analog of a central result of Lott-Sturm-Villani in optimal transport theory \cite{vil}. For details on notation and a detailed discussion on relationship with the literature, we refer to Section \ref{Sec: finite entropy section}.  

\begin{theorem}[Theorem \ref{FinEntGeodConvThm}]
If $\chi =\beta + i\ddbar f$ satisfies \eqref{eq: ChiProp}, then $(\textup{Ent}_\chi(X,\o),d_1)$ is a geodesic sub-metric space of $(\mathcal E^1(X,\o),d_1)$. Additionally, if $\Ric \o \geq \beta$ then the map $\textup{Ent}_\chi(X,\o) \ni u \to \textup{Ent}(e^{-f}\o^n,\o_u^n) \in \Bbb R$  is convex along finite energy geodesics.
\end{theorem}

\paragraph{Convergence and large time behaviour of the weak twisted Calabi flow.} As advertised above, using Theorem \ref{ExtKEnergyThmIntr}, we can run the weak twisted Calabi flow $[0,\infty) \ni t \to c_t \in \mathcal E^2$  for any starting point $c_0 \in \mathcal E^2$.  Indeed, $(\mathcal E^2,d_2)$ is a CAT(0)-space and the extended functional $\mathcal K_\chi$ is convex along $d_2$-geodesics, hence we are in the setting of \cite{may} as detailed in Section \ref{subsect: gradient flow}. 
This yields a flow of (possibly singular) K\"ahler potentials which is uniquely determined by the corresponding normalized Monge-Amp\`ere measures, which in turn yields a flow of probability measures which is regularizing in the sense that the entropy immediately becomes finite and in particular the measures have an $L^1$-density for positive times. 

When $\chi$ is smooth and $X$ is a  Riemann surface, the smooth twisted Calabi flow was recently explored in \cite{po}. To provide consistency, we will show that the weak twisted Calabi flow agrees with the smooth version whenever the latter exists (Proposition \ref{prop:consistence}), generalizing a result of Streets in case $\chi=0$ \cite{st2}. Providing additional consistency, as an application of  Theorem \ref{ExtKEnergyThmIntr}, in Section \ref{sec: weak twisted Calabi} we show that Street's a priori different minimizing movement Calabi flow coincides with our weak Calabi flow.

Generalizing twisted csc-K metrics, by  $\mathcal M^p_\chi$ we denote the minimizers of the extended K-energy on $\mathcal E^p$:
$$\mathcal M^p_\chi = \{u \in \mathcal E^p: \ \mathcal K_\chi(u) = \inf_{v \in \mathcal E^p}\mathcal K_\chi(v)\}.$$ 
In the case $\chi=0$ we will simply use $\mathcal M^p:=\mathcal M^p_0$. Concerning the convergence and blow-up behavior of the weak twisted Calabi flow we prove the following concrete result:

\begin{theorem}[Theorem \ref{thm: large time Calabi}\label{WeakCalFlowConv}] Suppose $(X,\o)$ is a compact connected K\"ahler manifold and $\chi=\beta + i\ddbar f$ satisfies \eqref{eq: ChiProp}. The following statements are equivalent:
\vspace{-0.1in}
\begin{itemize}
\setlength{\itemsep}{1pt}
    \setlength{\parskip}{1pt}
    \setlength{\parsep}{1pt}  
\item[(i)] $\mathcal M^2_\chi $ is nonempty.
\item[(ii)] For any weak twisted Calabi flow trajectory $t \to c_t$ there exists $c_\infty \in \mathcal M^2_\chi$ such that $d_1(c_t,c_\infty) \to 0$ and $\textup{Ent}(e^{-f}\o^n,\o_{c_t}^n) \to \textup{Ent}(e^{-f}\o^n,\o_{c_\infty}^n)$.
\item[(iii)] Any weak twisted Calabi flow trajectory $t \to c_t$ is $d_2$-bounded.
\item[(iv)] There exists a weak twisted Calabi flow trajectory $t \to c_t$ and $t_j \to \infty$ for which the sequence $\{ c_{t_j}\}_j$ is $d_2$-bounded.
\end{itemize}
\end{theorem}

$\bullet$ By the consistency result discussed above, the previous theorem in particular applies to the smooth Calabi flow (when it exists) and it should be stressed that the result and its elaborations discussed below are new also in this smooth case. In particular it generalizes results of the first author on the smooth Calabi flow on Fano manifolds without non-trivial holomorphic vector-fields \cite{brm1}. One new feature of our result is that the latter assumption, which guarantees the uniqueness of csc-K metrics, is not needed. This means that the limit $c_\infty$ is not uniquely determined by $X$ and will, in general, depend on the initial data $c_0$.

$\bullet$  By \cite[Theorem 5]{da2} and part (ii) of the above theorem, if a csc-K potential exists in $\mathcal H_\o$ then the weak Calabi flow $t \to c_t$ converges pointwise a.e. to some potential $c_\infty \in \mathcal M^2$, and the measures $\o^n_{c_t}$ converge weakly and in entropy to $\o_{c_\infty}^n$. In the Fano case it additionally follows that $c_\infty$ is csc-K. However, due to progress on the regularity Conjecture \ref{DRConj} discussed in the companion paper \cite{bdl2}, this result also holds on general K\"ahler manifolds as well, making further progress on Donaldson's conjecture. (see Theorem \ref{thm: DRannouncement} and Theorem \ref{thm: Cal_flow_announcement} below). 
 

$\bullet$  Finally, in light of Theorem \ref{thm: dweakchar}, we mention that Theorem \ref{WeakdConvProp}(ii) strengthens the corresponding convergence result of Streets in \cite{st2}. Given a CAT(0) metric space $(M,d)$, it is possible to introduce a notion of weak $d$-convergence, generalizing the concept of weak convergence on Hilbert spaces (Section \ref{subsect: CAT0}). In general, little concrete is known about this type of convergence \cite{kp}. Streets however observed that one can adapt the result of Ba\u c\'ak \cite{ba} to our setting, i.e., whenever $\mathcal M^2$ is non-empty, each weak Calabi flow trajectory converges $d_2$-weakly to an element of  $\mathcal M^2$ \cite{st2}. Though weak $d_2$ convergence does not even imply weak $L^1$ convergence of the potentials (Remark \ref{rem: weakd2convergence}), we use this idea in the proof of the above theorem together with the following result, which sheds light on the relationship between all the different topologies involved:
\begin{theorem}[Theorem \ref{d1weakd2dominatethm}]\label{thm: dweakchar}
	 Suppose $\{u_k\}_k \subset \mathcal E^2$  is $d_2$-bounded and $u \in \mathcal E^2$. Then $d_1(u_k,u) \to 0$ if and only if $\| u_j - u\|_{L^1(X)} \to 0$ and  $u_k$ converges to $u$ $d_2$-weakly.  
\end{theorem}
\paragraph{The conjectural picture of Donaldson.} Before we proceed, let us note a last corollary of Theorem \ref{WeakCalFlowConv}, a consequence of the equivalence between (i) and (iv):
\begin{corollary}\label{CalabiFlowAlternativeCor}
Suppose $(X,\o)$ is a compact connected K\"ahler manifold and $[0,\infty) \ni t \to c_t \in \mathcal E^2$ is a weak twisted Calabi flow trajectory. Exactly one of the following holds:
\vspace{-0.1in}
\begin{itemize}
\setlength{\itemsep}{1pt}
    \setlength{\parskip}{1pt}
    \setlength{\parsep}{1pt}  
\item[(i)] The curve $t \to c_t$ $d_1$-converges to some $c_\infty \in \mathcal M^2_\chi$.
\item[(ii)] $d_2(c_0,c_t) \to \infty$ as $t \to \infty.$
\end{itemize}
\end{corollary}

 Though this corollary is in line with Donaldson's conjectural picture, one would like to understand how a diverging Calabi flow trajectory 'destabilizes' the K\"ahler structure, as proposed in  Conjecture \ref{DonaldsonConj}. In this direction we recall the following concept from \cite{dh}: suppose $(M,d)$ is a geodesic metric space and $[0,\infty) \ni t\to \gamma_t \in  M$ is a continuous curve. We say that the unit speed $d$--geodesic ray $[0,\infty) \ni t\to g_t \in  M$ is $d$-\emph{weakly asymptotic} to the curve $t \to \gamma_t$, if there exists $t_j \to \infty$ and unit speed $d$--geodesic segments $[0,d(\gamma_0,\gamma_{t_j})] \ni t\to g^j_t \in  M$ connecting $\gamma_0$ and $\gamma_{t_j}$ such that $\lim_{j \to\infty}d(g^j_t,g_t)=0, \ t \in [0,\infty).$

Clearly, to have a geodesic ray weakly asymptotic to $t \to \gamma_t$, we need $t \to d(\gamma_0,\gamma_t)$ to be unbounded.  By the above corollary, this condition makes diverging weak Calabi flow trajectories $t\to c_t$ perfect candidates for this construction. However, more needs to be known about $t \to c_t$ before we can proceed. In \cite[Conjecture 2.8]{dr} it was pointed out that an important roadblock in resolving  Tian's properness conjecture for csc-K metrics is a conjecture about regularity of minimizers of $\mathcal K$. The twisted version of this conjecture should also hold:
\begin{conjecture} \cite{dr} \label{DRConj} Suppose $(X,\o)$ is a compact connected K\"ahler manifold and $\chi$ is smooth. Then $\mathcal M^1_\chi \subset \mathcal H_\o$, i.e. $\mathcal M^1_\chi$ contains only smooth twisted csc-K potentials.
\end{conjecture}
We note that this conjecture generalizes an earlier conjecture of Chen about $C^{1,1}$ minimizers of $\mathcal K$ \cite[Conjecture 6.3]{c3}. When $(X,\o)$ is Fano, Conjecture \ref{DRConj} was proved in \cite{brm1,bbegz}. The next result partially confirms Donaldson's conjecture in the Fano case and also in the case when $\chi$ is a K\"ahler form.

\begin{theorem}[Theorem \ref{thm: analog of Donaldson conjecture}]\label{thm: Donaldson 1}  Suppose $(X,\o)$ is a compact connected K\"ahler manifold, $\chi\geq 0$ is smooth and Conjecture \ref{DRConj} holds. Let $[0,\infty) \ni t \to c_t \in \mathcal E^2$ be a weak twisted Calabi flow trajectory. Exactly one of the following holds: 
\vspace{-0.05in}
\begin{itemize}
\setlength{\itemsep}{1pt}
    \setlength{\parskip}{1pt}
    \setlength{\parsep}{1pt}  
\item[(i)] The curve $t \to c_t$ $d_1$-converges to a smooth twisted csc-K potential $c_\infty$.
\item[(ii)] $d_1(c_0,c_t) \to \infty$ as $t \to \infty$ and the curve $t \to c_t$ is $d_1$-weakly asymptotic to a finite energy geodesic $[0,\infty) \ni t \to u_t \in \mathcal E^1$ along which $\mathcal K_\chi$ decreases.
\end{itemize}
\vspace{-0.05in}
If  $\chi>0$, then independently of Conjecture \ref{DRConj} exactly one of the following holds:
\vspace{-0.05in}
\begin{itemize}
\setlength{\itemsep}{1pt}
    \setlength{\parskip}{1pt}
    \setlength{\parsep}{1pt}  
\item[(i')] The curve $t \to c_t$ $d_1$-converges to a unique minimizer in $\mathcal E^1$ of $\mathcal K_\chi$.
\item[(ii')] $d_1(c_0,c_t) \to \infty$ as $t \to \infty$ and the curve $t \to c_t$ is $d_1$-weakly asymptotic to a finite energy geodesic $[0,\infty) \ni t \to u_t \in \mathcal E^1$ along which $\mathcal K_\chi$ strictly decreases.
\end{itemize}
\end{theorem}
Though stated differently, when $(X,\o)$ is Fano and $\chi=0$ the analog of this result for the K\"ahler-Ricci flow has been obtained in \cite[Theorem 2]{dh}. There we have smooth convergence in (i) and the along the geodesic ray of (ii) the potentials are  bounded, all thanks to the Perelman estimates available for the K\"ahler-Ricci flow. It would be interesting to compare the above theorem to the results in \cite{cs}, where the authors construct in a specific situation a geodesic ray asymptotic to the Calabi flow  and are able to draw geometric conclusions based on this.

\paragraph{Concluding remarks and additional results.} Based on geometric considerations, and the analogous picture in case of the K\"ahler-Ricci flow \cite{gz2}, it is natural to speculate that for any starting point $c_0 \in \mathcal E^2$, the weak Calabi flow $t \to c_t$ is instantly smooth, i.e. $c_t \in \mathcal H_\o, \ t >0$ (see also \cite[Conjecture 3.5]{c4}). Such a result would instantly give the $\mathcal E^2$ version of Conjecture \ref{DRConj}, that $\mathcal E^2$-minimizers of $\mathcal K$ are smooth csc-K metrics. Indeed, by the general result of Mayer \cite{may}, the weak Calabi flow $t \to c_t$ starting at a minimizer $c_0 \in \mathcal E^2$ has to be stationary. If $t \to c_t$ was instantly smooth, then we could conclude that $c_0 \in \mathcal H_\o$.

In the companion paper \cite{bdl2} we make progress on Conjecture \ref{DRConj} using different techniques from the ones presented in this paper:

\begin{theorem}\label{thm: DRannouncement}Suppose $(X,\o)$ is a K\"ahler manifold and $\mathcal H_\o$ contains a csc-K potential. Then $\mathcal M^1$ contains only smooth csc-K potentials.
\end{theorem} 

The consequences of this theorem related to K-stability and energy properness will be discussed in \cite{bdl2}. As $\mathcal M^2 \subset \mathcal M^1$, here we just mention the following consequence of this result and Theorem \ref{WeakCalFlowConv}(ii), making further progress on Conjecture \ref{DonaldsonConj} (see also \cite[Remark 1.10]{st2}):

\begin{theorem}\label{thm: Cal_flow_announcement} Suppose $(X,\o)$ is a K\"ahler manifold and $\mathcal H_\o$ contains a csc-K potential $u$. Then any weak Calabi flow trajectory $t \to c_t$ $d_1$-converges to a smooth csc-K potential $c_\infty \in \mathcal H_\o$. In addition, the densities $\o_{c_t}^n/\o^n$ converge in $L^1$ to the density $\o_{c_\infty}^n/\o^n$.
\end{theorem}

\vspace{-.15in}
\paragraph{Acknowledgments.} We would like to thank L\'aszl\'o Lempert for his careful suggestions regarding the presentation of the paper.
\vspace{-.15in}
\paragraph{Organization of the paper.} In the first part of Section \ref{sect: preliminaries} we recall recent results on  complex Monge-Amp\`ere theory which we will use in this paper. In the second part we briefly recall  Mayer's theory of gradient flows on non-positively curved metric spaces. The approximation of finite energy $\omega$-plurisubharmonic functions with convergent entropy is presented in Section \ref{sect: approximation}. The  twisted Mabuchi energy is studied in  Section \ref{sect: twisted K-energy}. The weak $d_2$ topology is explored in Section \ref{sec: weak d2 convergence}, while the last section is devoted to the weak twisted Calabi flow.  

\section{Preliminaries}\label{sect: preliminaries}
\subsection{The twisted K-energy}

Suppose $\chi$ is a closed positive $(1,1)$-current and $\beta$ is a smooth closed (1,1)-form in the same cohomology class as $\chi$. In most applications of K\"ahler geometry, the twisting current $\chi$ can be smooth, but in order to treat the case of smooth and singular canonical metrics (e.g. conical csc-K metrics) together, it is natural to ask for the following more general restriction on $\chi$:
\begin{equation}\label{eq: ChiProp}
\chi = \beta + i\ddbar f, \textup{ where } f \in \textup{PSH}(X,\beta) \textup{ with }\ e^{-f} \in L^{1}(X,\o^n).
\end{equation}
We observe that the integrability condition $e^{-f}\in L^1(X,\omega^n)$ implies that $e^{-f}\in L^p(X,\omega^n)$ for some $p>1$ as follows from the openness conjecture, recently proved  by Berndtsson (\cite{Bern13}, see also \cite{gzh}).  We note that some of our results, in particular Theorem \ref{ExtKEnergyThmIntr} above, hold for more general $\chi$. However, it is unlikely that greater generality will have applications, and we leave it to the  reader to find optimal conditions for $\chi$ in our theorems.

The twisted K-energy $\mathcal K_\chi:\mathcal H_\o \to \Bbb R$ can now be defined as follows:
\begin{equation}\label{KEnTwistDef}
\mathcal K_\chi(u) = \textup{Ent}(e^{-f}\o^n,\o^n_u)
+ \bar{S}_\chi\textup{AM}(u) - n\textup{AM}_{\Ric \o-\beta}(u) - \int_X f\o^n,
\end{equation}
where $\bar{S}_\chi = nV^{-1}\int_X (\Ric_\o - \chi) \wedge \o^{n-1}$. Notice that 
for $\beta =0, f=0$ we get back the usual $K$-energy \eqref{KEnergyDef}. Using the identity $n\textup{AM}_\chi(u)=n\textup{AM}_\beta(u) +\int {f} \o_u^n - \int f \o^n$ one can give an alternative formula for $\mathcal K_\chi$, perhaps more familiar from the literature:
\begin{equation}\label{KEnTwistAltDef}
\mathcal K_\chi(u) = \textup{Ent}(\o^n,\o^n_u)
+ \bar{S}_\chi\textup{AM}(u) - n\textup{AM}_{\Ric \o-\chi}(u).
\end{equation}
The virtue of this formula is that it shows that $\mathcal K_\chi$ is independent of the choice of $\beta$ and $f$. As it will be clear shortly, when trying to extend $\mathcal K_\chi$, our original definition is more advantageous however. Note that when $\chi$ is smooth, the first order variation of $\mathcal K_\chi$ is given by the following formula:
$$\langle D\mathcal K_\chi(u), \delta v\rangle = V^{-1} \int_X \delta v (\bar S_\chi-S_{\o_u} + \textup{Tr}^{\o_u}\chi)\o^n_u.$$
Hence, the critical points of this functional are the twisted csc-K potentials, as these satisfy $\bar S_\chi-S_{\o_u} + \textup{Tr}^{\o_u}\chi=0$. The smooth twisted Calabi flow is defined analogously.

\label{sec: Prelim the twisted K-energy}
\subsection{The Complete Geodesic Metric Spaces $(\mathcal E^p,d_p)$}\label{subsect: the space Ep}
In this section we mostly summarize results from \cite{da1,da2,begz, bbegz} needed the most in this paper. Formula \eqref{FinslerDef} introduces $L^p$-type weak Finsler metrics on the Fr\'echet manifold $\mathcal H_\o$. 
A curve $[0,1]\ni t \to \alpha_t \in \mathcal{H}_\o$ 
is called smooth if $\alpha(t,z)=\alpha_t(z) \in C^\infty([0,1] \times M)$. The $L^p$-length of a smooth curve $t \to \alpha_t$ is given by 
\begin{equation*}\label{curve_length_def}
l_p(\alpha):=\int_0^1\|\dot \alpha_t\|_{p,\alpha_t}dt.
\end{equation*}
\begin{definition}
The path length pseudo-distance of $(\mathcal H_\o,d_p)$ is defined by
$$
d_p({u_0},{u_1}):= 
\inf\{l_p(\alpha)\,:\,[0,1]\ni t \to \alpha_t\in \mathcal H_\o \textup{\ is a smooth curve with \ }
\alpha_0=u_0,\, \alpha_1=u_1\}.
$$
\end{definition}

It turns out $d_p$ is an honest metric
\cite[Theorem 3.5]{da2}. To state the result,
consider 
$[0,1]\times\RR\times X$ as a complex manifold of dimension
$n+1$, and denote by $\pi_2:[0,1]\times\RR\times X\ra X$ the natural projection.
\begin{theorem}
\label{d1Thm}
$(\mathcal H_\o, d_p)$ is a metric space.
Moreover for any $t \in [0,1]$, 
\begin{equation*}
\label{distgeod}
d_p(u_0,u_1)=\|\dot u_t \|_{u_t}\ge0,
\end{equation*}
where $\dot u_t = d u_t/dt$ is the 'tangent' at time $t$ of $t \to u_t$, the $\Bbb R$-invariant solution of the \MA equation,
\begin{equation}\label{MabuchiEq}
\varphi\in\textup{PSH}(\pi_2^\star\o, [0,1]\times\RR\times X),\ 
(\pi_2^\star\o+\i\ddbar \varphi)^{n+1}=0, \ 
 \varphi|_{\{i\}\times\RR}=u_i,\; i=0,1.
\end{equation}
\end{theorem}
Some comments are in order. We note that by the main result of \cite{c} (see also \cite{bl1}), the equation \eqref{MabuchiEq} has a unique $\Bbb R$-invariant solution for which $u(t,x)=u_t(x)$ has bounded Laplacian in $[0,1]\times \Bbb R \times X$. We can look at this solution as a curve
$$[0,1]\ni t \to u_t \in \mathcal H^{\Delta}_\o=\{u \in \textup{PSH}(X,\o),  \ \Delta_{\o}u \in L^\infty(X)\}.$$
We call this curve the \emph{weak geodesic} connecting $u_0,u_1 \in \mathcal H_\o$.
Recall that
$$
\textup{PSH}(X,\o)=\{ \varphi \in L^1(X,\o^n): \varphi \textup{ is usc and } \o_\varphi\geq 0 \}.
$$
Given $\varphi_k\in \textup{PSH}(X,\o), \ k=1\ldots n$,  one can introduce the following \emph{non-pluripolar product} \cite{begz}, generalizing the Beford-Taylor
product \cite{bt} concerning the case with bounded potentials:
\begin{equation}\label{eq: nonpluripolar product}
\o_{\varphi_1} \wedge \o_{\varphi_2} \wedge \ldots \wedge \o_{\varphi_n}:=\lim_{j\ra-\infty}
{\bf 1}_{\cap_k\{\varphi_k>j\}}
\o_{\max(\varphi_1,j)} \wedge \o_{\max(\varphi_2,j)} \wedge \ldots \wedge \o_{\max(\varphi_n,j)}.
\end{equation}
The measures $\o_{\max(\varphi_1,j)}  \wedge \ldots \wedge \o_{\max(\varphi_n,j)}$ are defined by the
work of Bedford--Taylor \cite{bt} since $\max\{\varphi,j\}$ is bounded. Restricted to $\cap_k\{\varphi_k>j\}$ these measures are increasing, hence the above limit is well defined \cite{gz1,begz} and $\int_X \o_{\varphi_1} \wedge \ldots \wedge \o_{\varphi_n} \leq \int_X \o^n$.

Following Guedj-Zeriahi \cite[Definition 1.1]{gz1} we introduce the class of potentials with ``full volume" $
\mathcal E(X,\o):=
\big\{
\varphi\in \textup{PSH}(X,\o):\int_X \o_{\varphi}^n= \int_X \o^n
\big\}$ and the corresponding finite energy classes:
$$
\mathcal E^p:=\big\{\varphi\in\mathcal E(M,\omega)\,:\, \int|\varphi|^p\o^n_{\varphi}<\infty\big\}.
$$
The next result characterizes the $d_p$-metric completion of $\mathcal H_\o$:
\begin{theorem}\textup{\cite[Theorem 2]{da2}}
\label{dpCompletionThm}
The metric completion of $(\mathcal H_\o,d_p)$ equals $(\mathcal E^p,{d_p})$,
where 
$$
d_p(u_0,u_1):=\lim_{k\ra\infty}
d_p(u_0^k,u_1^k),
$$
for any smooth decreasing sequences $\{u_i^k \}_{k\in\Bbb N}\subset\mathcal H_\o$
converging pointwise to $u_i \in \mathcal E^p, i=0,1$.
Moreover, for each $t\in(0,1)$, define
\begin{equation*}\label{EpGeodDef}
u_t:= \lim_{k \to\infty}u_t^k, \ t \in (0,1),
\end{equation*}
where $u_t^k$ is the weak geodesic connecting $u_0^k, u_1^k$. Then $u_t\in \mathcal E^p$, the curve $[0,1] \ni t \to u_t \to \mathcal E^p$ is well-defined independently of the choices
of approximating sequences and is a $d_p$-geodesic.
\end{theorem}

Note that by \cite{Dem}, \cite{bk} it is always possible to find approximating sequences $\{u_0^k\}_k, \{u_1^k\}_k$ as in the above theorem. We now recall \cite[Theorem 3]{da2}, giving a concrete characterization of the growth of all $d_p$ metrics:

\begin{theorem}\label{thm: comparison d2 and I2}
There exists $C>1$ such that for all $u,v\in\mathcal E^p$,
\begin{equation}
\label{dpCharFormula}
C^{-1} d_p(u,v) \leq \bigg(\int_X |u-v|^p\o_u^n\bigg)^{\frac{1}{p}} + \bigg(\int_X |u-v|^p\o_v^n\bigg)^{\frac{1}{p}}
\le C d_p(u,v).
\end{equation}
\end{theorem}
The inequalities in (\ref{dpCharFormula}) have an important consequence: $|\sup_X u|\leq C d_p(u,0)$ for all $u\in \Ec^p$. Also, when $p=1$, $d_1$-convergence is  equivalent to convergence with respect to the  quasi-distance  $I(u,v)=\int_X (u-v)(\omega_v^n-\omega_u^n)$ introduced in \cite{bbegz} as shown in \cite[Theorem 5.5]{da2}. 

Monotonic sequences behave well with respect to all $d_p$-metrics \cite[Proposition 4.9]{da2}:

\begin{proposition}\label{prop: monotone implies dp convergence} Suppose $u_k,u \in \mathcal E^p$. If $\{u_k\}_k$ is monotone decreasing/increasing and converges to $u$ a.e. then $d_p(u_k,u) \to 0$.
\end{proposition}

Given $u_0,u_1,\ldots,u_k \in \textup{PSH}(X,\o)$, by $P(u_0,u_1,\ldots,u_k) \in \textup{PSH}(X,\o)$ we denote the following upper envelope:
$$P(u_0,u_1, \ldots,u_k) = \sup \{ v \in \textup{PSH}(X,\o) \textup{ such that } v \leq u_0, \ldots , v\leq u_k\}.$$

According to the next proposition it is possible to sandwich a subsequence of any $d_p$-convergent sequence between two monotone sequences converging to the same limit.

\begin{proposition}\label{MonotoneDominationProp}. Suppose $u_k,u \in \mathcal E^p$. If $d_p(u_k,u) \to 0$ then there exists a subsequence $k_j \to \infty$ and $\{w_{k_j}\}_j \subset \mathcal E^p$ decreasing, $\{v_{k_j}\}_j \subset \mathcal E^p$ increasing with $v_{k_j} \leq u_{k_j} \leq w_{k_j}$ and $d_p(w_{k_j},u),d_p(v_{k_j},u) \to 0 $.
\end{proposition}

\begin{proof}By \eqref{dpCharFormula} there exists $C>0$ such that $|\sup_X u_j| \leq C, \  j \geq 1.$ We introduce the following sequence:
$$w_k = \textup{usc}\Big({\sup_{j \geq k} u_j}\Big).$$
As $u_k \leq w_k \leq C$, by \cite{gz1} it follows $w_k \in \mathcal E^p$. As $d_p(u_k,u) \to 0$, we have that $u_k \to u$ pointwise a.e., hence $w_k$ decreases  to $u$. Proposition \ref{prop: monotone implies dp convergence} then gives $d_p(w_k,u) \to 0$.

Now we construct the increasing sequence $v_{k_j}$. To do this first take a subsequence $u_{k_j}$ of $u_k$ satisfying $d_p(u_{k_j},u) \leq 2^{-j}$. As follows from the proof of \cite[Theorem 4.17]{da2} and \cite[Theorem 9.2]{da1}, the following limit exists
$$
v_{k_j} = P(u_{k_j},u_{k_{j+1}},u_{k_{j+2}},\ldots) =: \lim_{h \to \infty}P(u_{k_j},u_{k_{j+1}},\ldots, u_{k_{j+h}})
$$
Additionally, $\{v_{k_j}\}_j \subset \mathcal E^p$ and and $v_{k_j}$ increases a.e. to $u$. The previous proposition now gives $d_p(u,v_{k_j})\to 0$.
\end{proof}

Though stated differently, the next proposition is essentially contained contained in \cite{begz}:

\begin{proposition} \label{d2weaklimits} Suppose $p \geq 1$, $\{ u_j\}_j \subset \mathcal E^p$ is a $d_p$--bounded sequence and $u \in \textup{PSH}(X,\o)$ with $\| u_j -u\|_{L^1(X,\o^n)}\to 0$. Then $u \in \mathcal E^p$.
\end{proposition}
\begin{proof}  $d_p$--boundedness implies that $|\sup_X u_j| \leq B$ for some $B \in \Bbb R$ \eqref{dpCharFormula}. For simplicity  assume that $B=0$. The following sequence converges a.e. to $u$:
$$w_k = \textup{usc}\left(\sup_{j \geq k}u_{j}\right) \leq 0.$$
This sequence is additionally decreasing and because $u_k \leq w_k \leq 0$, we have that $w_k \in \mathcal E^p$. If we could argue that $\{ w_k\}_k$ is uniformly $d_p$-bounded then we would be finished by \cite[Lemma 4.16]{da2}. But $d_p$-boundedness follows from \eqref{dpCharFormula}. Indeed $\int_X |w_k|^p \o^n \leq \int_X |u_k|^p \o^n$ and $\int_X |w_k|^p \o^n_{w_k} \leq C(p)\int_X |u_k|^p \o^n_{u_k}$ by \cite[Lemma 3.5]{gz1}, hence by \eqref{dpCharFormula}
the quantity $d_p(0,w_k)$ is uniformly bounded.
\end{proof}

Given two Borel measures $\mu,\nu$ on $X$, if $\nu$ is not subordinate to $\mu$, then by definition $\textup{Ent}(\mu,\nu)=\infty$. On the other hand, if $\nu$ is subordinate to $\mu$ then $\textup{Ent}(\mu,\nu)=\int_X \log(f) \nu$, where $f$ is the Radon-Nikodym density of $\nu$ with respect to $\mu$. The entropy functional $\mu \to \textup{Ent}(\mu,\nu)$ is lsc with respect to weak convergence of measures \cite{dz}. Related to entropy we recall the following crucial compactness result  \cite[Theorem 2.17]{bbegz}:

\begin{theorem} 
\label{EntropyCompactnessThm}
Let $p>1$ and 
suppose $\mu = f \o^n$ is a probability measure with 
$f \in L^p(X,\o^n)$.  
Suppose there exists $C>0$ such that $\{u_k\}_k \subset \mathcal E^1$ satisfies
$$
|\sup_X u_k| < C, \  \textup{Ent}(\mu,\o_{u_k}^n) < C. 
$$
Then $\{ u_k\}_k$ contains a $d_1$-convergent subsequence. 
\end{theorem}

\subsection{The complex Monge-Amp\`ere equation in $\Ec^p$}
We summarize in this section basic results concerning solutions of degenerate complex Monge-Amp\`ere equations. 

A subset $E\subset X$ is called pluripolar if it is contained in the singular set of a function $\varphi\in \psh(X,\omega)$, i.e. $E\subset \{\varphi=-\infty\}$. Let $\mu$ be a positive measure on $X$  with total mass $\mu(X)=\int_X \omega^n$. We consider the following equation 
\begin{equation}\label{eq: MA Ep}
	\omega_\varphi^n = \mu. 
\end{equation} 
It was proved in \cite{gz1}, by approximation with the smooth case established in the seminal work of Yau and in \cite{bbgz} by a direct variational approach that when $\mu$ does not charge pluripolar sets the equation (\ref{eq: MA Ep}) has a solution  $\varphi\in \Ec(X,\omega)$. The solution turns out to be unique up to an additive constant \cite{Diw}. For each $\vep>0$, the same variational approach as in \cite{bbgz} applied for the functional 
$$
F_{\vep}(u) := {\rm AM} (u) -\frac{1}{\vep} \log \int_X e^{\vep u} d\mu, \ u\in \Ec^1,
$$
shows that there exists a solution  $\varphi_{\vep}\in \Ec^1$ to the equation
\begin{equation}\label{eq: MA Ep epsilon}
	\omega_{\varphi_{\vep}}^n = e^{\vep \varphi_{\vep}}\mu. 
\end{equation} 
The solution is uniquely determined as follows from the comparison principle (see \cite[Proposition 4.1]{bg}). The following version of the comparison principle will be useful later. 
\begin{lemma}
	\label{lem: comparison principle in E}
	Let $\vep>0$. Assume that $\varphi\in \Ec(X,\omega)$ is a solution of (\ref{eq: MA Ep epsilon}) while $\psi\in \Ec(X,\omega)$ is a subsolution, i.e. $\omega_{\psi}^n \geq e^{\vep \psi}\mu.$
	Then $\varphi\geq \psi$ on $X$. 
\end{lemma} 
This result might be well known to experts in Monge-Amp\`ere theory. As a courtesy to the reader we give a proof below.
\begin{proof}
	By the comparison principle for the class $\Ec(X,\omega)$ (see \cite{gz1}) we have 
	$$
	\int_{\{\varphi<\psi\}} \omega_{\psi}^n \leq \int_{\{\varphi<\psi\}} \omega_{\varphi}^n.
	$$
	As $\varphi$ is a solution and $\psi$ is a subsolution to (\ref{eq: MA Ep epsilon}) we also have
	$$
	\int_{\{\varphi<\psi\}} e^{\vep \psi} d\mu \leq \int_{\{\varphi<\psi\}} \omega_{\psi}^n \leq \int_{\{\varphi<\psi\}} \omega_{\varphi}^n = \int_{\{\varphi<\psi\}} e^{\vep \varphi} d\mu\leq \int_{\{\varphi<\psi\}} e^{\vep \psi} d\mu.
	$$
	It follows that all inequalities above are equalities, hence $\varphi\geq \psi$ $\mu$-almost everywhere on $X$. By Dinew's domination principle \cite[Propostion 5.9]{blle} we get $\varphi\geq \psi$ everywhere on $X$.
\end{proof}
One might wonder whether the solution of (\ref{eq: MA Ep}) arises as a limit of solutions of (\ref{eq: MA Ep epsilon}) as $\vep\to 0$. The following result (a version of which is stated as an exercise in \cite{GZbook}) answers this affirmatively. 
\begin{lemma}\label{lem: epsilon equation limit}
	Let $p\geq 1$. Assume that $\mu= (\omega+i\ddbar \varphi)^n$ with $\varphi\in \Ec^p$ and $\int_X \varphi d\mu =0$. For each $\vep>0$, let $\varphi_{\vep}\in \Ec^1$ be the unique  solution to (\ref{eq: MA Ep epsilon}). Then in fact $\varphi_{\vep}\in \Ec^p$ and $d_p(\varphi_{\vep},\varphi)\to 0$ as $\vep \to 0$.
\end{lemma}
\begin{proof}
	As $\varphi-\sup_X \varphi$ is a subsolution of (\ref{eq: MA Ep epsilon}), it follows from Lemma \ref{lem: comparison principle in E} that $\varphi_{\vep}\geq \varphi-\sup_X \varphi, \ \forall \vep>0$, hence $\varphi_{\vep}\in \Ec^p$. We claim that $\varphi_{\vep}$ is uniformly bounded from above for $\vep\in [0,1]$. Assume on the contrary that we can extract a subsequence denoted by $\varphi_j=\varphi_{\vep_j}$ such that $\sup_X \varphi_j\to +\infty$. The sequence $\psi_j:=\varphi_j -\sup_X \varphi_j$ stays in a compact set in $L^1(X,\omega^n)$, hence a subsequence (still denoted by $\varphi_j$) converges  to some $\psi\in \psh(X,\omega)$. It then follows that $\varphi_j=\psi_j + \sup_X \varphi_j$ converges uniformly to $+\infty$. In the other hand, by Jensen's inequality (for simplicity we may assume that $\mu(X)=1$) we have
	$$
	\int_X \varphi_j d\mu \leq 0.
	$$
	Since $\varphi_j$ is bounded from below by $\varphi-\sup_X \varphi$, which is integrable with respect to $d\mu$, the above inequality contradicts with the fact that $\varphi_j$ converges uniformly to $+\infty$. Hence the claim follows. 
	
	Now the family $\varphi_{\vep}$ stays in a compact set of $L^1(X,\omega^n)$. As $\vep \to 0$ each cluster point $\varphi_0$   satisfies 
	$$
	\omega_{\varphi_{0}}^n \geq \left(\liminf_{\vep \to 0} e^{\vep \varphi_{\vep}} \right) \mu =\mu,
	$$
	as follows from \cite[Corollary 2.21]{begz}. As the two measures have the same total mass, one obtains equality. That $\varphi_0=\varphi$ follows from uniqueness of complex Monge-Amp\`ere measures and the following identity:
	$$
	0=\lim_{\vep\to 0} \frac{1}{\vep} \log \int_{X} e^{\vep \varphi_{\vep}} d\mu = \int_X \varphi_0 d\mu.   
	$$
	The last statement follows from \eqref{dpCharFormula}, \eqref{eq: MA Ep epsilon} and the dominated convergence theorem. 
\end{proof}

\subsection{Weak convergence in a CAT(0) space}\label{subsect: CAT0}

Let us recall that a geodesic metric space $(M,d)$ is a metric space for which any two points can be connected with a geodesic. By a geodesic connecting two points $a,b \in M$ we understand a curve $\alpha: [0,1]\to M$ such that $\alpha(0)=a$, $\alpha(1)=b$ and
\begin{equation*}
d(\alpha(t_1),\alpha(t_2))=|t_1 -t_2|d(a,b),
\end{equation*}
for any $t_1,t_2 \in [0,1]$.
Furthermore, a geodesic metric space $(M,d)$ is non-positively curved (in the sense of Alexandrov) or CAT(0) if for any distinct points $q,r \in M$ there exists a geodesic $\gamma:[0,1] \to M$ joining $q,r$ such that for any $s \in \{ \gamma \}$  and $p \in M$ the following inequality is satisfied:
\begin{equation*}
d(p,s)^2 \leq \lambda d(p,r)^2 +(1-\lambda) d(p,q)^2 - \lambda(1-\lambda)d(q,r)^2,
\end{equation*}
where $\lambda = d(q,s)/ d(q,r)$. A basic property of CAT(0) spaces is that geodesic segments joining different points are unique. For more about these spaces we refer to \cite{bh}.

Let $\{x_n\}_n$ be a bounded sequence in a CAT(0) metric space $(M,d)$. For $x \in M$, we set
$$r(x,\{x_n\}_n) = \limsup d(x, x_n).$$

The \emph{asymptotic radius} $\{x_n\}_n$ is given by $r(\{x_n\}_n) = \inf\{r(x,\{x_n\}_n) : x \in M\},$
and the \emph{asymptotic center} $A(\{x_n\})$ of ${x_n}$ is the set
$$A(\{x_n\}_n) = \{x \in M : r(x,\{x_n\}_n) = r(\{x_n\}_n)\}.$$
It is well known (see, e.g., \cite[Lemma 4.3]{st2}) that in a CAT(0) space, $A(\{x_n\}_n)$ consists of exactly one point. A sequence $\{ x_n\}_n$ \emph{converges $d$-weakly} to $x \in M$, if $x$ is the asymptotic center of all subsequences of $\{ x_n\}_n$. The next result collects some facts about weak $d$-convergence.

For a more detailed account on weak $d$-convergence we refer to \cite{kp}, and for results related to the Calabi flow, see \cite[Section 4]{st2}.
If $(M,d)$ is a Hilbert space then weak $d$-convergence is the same as weak convergence in the sense of Hilbert spaces. With this in mind the contents of the next result may seem less surprising:

\begin{proposition}\label{WeakdConvProp} Suppose $(M,d)$ is a CAT(0) space. The following hold:
\vspace{-0.1in}
\begin{itemize}
\setlength{\itemsep}{1pt}
    \setlength{\parskip}{1pt}
    \setlength{\parsep}{1pt}  
\item[(i)] \cite[Proposition 3.5]{kp} If $\{ x_n\}_n$ is a $d$-bounded sequence then it has a weak $d$-convergent subsequence.
\item[(ii)] \cite[Proposition 3.2]{kp} Suppose $C \subset M$ is a geodesically convex closed set and $\{ x_n\}_n \subset C$ converges $d$-weakly  to $x \in M$. Then $x \in C$.
\end{itemize}
\end{proposition}

\subsection{General weak gradient flows}\label{subsect: gradient flow}

Let  $G$ be a $d$-lsc function
on a complete metric space $(M,d).$ In this generality there are,
as explained in \cite{ags}, various notions of weak gradient flows
$u_{t}$ for $G,$ emanating from an initial point $u_{0}$ in $M.$
A natural approximation scheme (the so called Minimizing Movement)
for obtaining such a candidate $t \to c_{t}$ was introduced by De Giorgi.
It can be seen as a variational formulation of the (backward) Euler
scheme: given $t\in[0,\infty)$ and a positive integer $m$ one first
defines a discrete version $c_{t}^{m}$ of $c_{t}$ as the $m$th step in the following ($m$-dependent) iteration with initial data $c^{m,0}_t=c_{0}$: given
$c^{m,j}_t\in Y$ the next step $c^{m,j+1}_t$ is obtained by minimizing the
following functional on $Y$ 
\begin{equation}\label{eq: minimizing movement eq}
v\to\frac{1}{2}d(v,c^{m,j}_t)^{2}+\frac{t}{m}G(v).
\end{equation}
If such a minimizer always exists then the corresponding Minimizing
Movement $c_{t}$ is defined as the large $m$ limit of $c_{t}^{m}=c_{t}^{m,m},$
if the limit exists in $(M,d).$ As shown by Mayer \cite{may}, if $(M,d)$ is a CAT(0) metric space and $G$ is convex this procedure indeed produces a unique limit $c_{t}$ with a number of useful properties. 
\begin{theorem}
\label{MayerThm}(\cite{may}) If $(M,d)$ is CAT(0), $G$ is a $d$-lsc convex function on $(M,d),$ then for any initial
point $c_{0}$ with  $G(c_0) < \infty$
the corresponding minimizing
movement $t \to c_{t}$ exists and defines a contractive continuous semi-group
(which is locally Lipchitz continuous on $[0,\infty)$). 
\end{theorem}
Moreover, as shown in \cite{may}, the curve $t \to c_{t}$ can be thought
of as the curve of steepest descent with respect to $G$ in the sense
that 
\begin{equation}
-\frac{d}{dt}(G(c_{t}))=\left|(\partial G)(c_{t})\right|\left|\frac{dc_{t}}{dt}\right|,\,\,\,\,\,\,\,\left|\frac{dc_{t}}{dt}\right|=|(\partial G)(c_{t})|\label{eq:relations equi to gf}
\end{equation}
for almost every $t,$ where $\left|(\partial G)(y)\right|$ is the
local upper gradient of $G$ at $y$ and $\left|\frac{dc_{t}}{dt}\right|$
is the metric derivative of $t \to c_{t}$ at $t$ (in the sense of \cite{ags}):
\begin{equation*}
\left|(\partial G)(y)\right|:=\limsup_{z\rightarrow y}\frac{\left(G(y)-G(z)\right)^{+}}{d(y,z)},\,\,\,\,\,\,\,\left|\frac{dc_{t}}{dt}\right|=\lim_{s\rightarrow t}\left|\frac{d(c_{s},c_{t})}{s-t}\right|\label{eq:def of local upper gradient}
\end{equation*}
In the case when $(M,d)$ is a finite dimensional Riemannian manifold and $G$ is smooth the relations \eqref{eq:relations equi to gf} are equivalent to the usual gradient flow formulation for $G.$ In the terminology of \cite{ags} the relations \eqref{eq:relations equi to gf} imply that the Minimizing Movement $t \to c_{t}$ provided by Mayer's theorem is a curve of\emph{ maximal slope} with respect to the upper gradient $\left|(\partial G)\right|$ (see \cite[Definition 1.3.2]{ags}). Moreover, by \cite[Theorem 4.0.4]{ags} the curve $t \to c_{t}$ is the unique solution of the following \emph{evolution variational inequality:}
\begin{equation}
\frac{1}{2}\frac{d}{dt}d^{2}(c_{t},v)\leq G(v)-G(c_t),\,\,\,\,\mbox{a.e.\,\,\ensuremath{t>0,\,\,\,\forall v:\, G(v)<\infty}}\label{eq:evi}
\end{equation}
among all locally absolutely continuous curves in $(M,d)$ such that 
$\lim_{t\rightarrow0}c_t=c_{0}$. Among other things, the above inequality shows that 
$$\lim_{t \to \infty}G(c_t)= \inf_{y \in M}G(y).$$
\begin{remark}\label{remark: initial potential}
A necessary condition for the solvability of the minimization steps \eqref{eq: minimizing movement eq} is to have $G(c_0) < \infty$. An approximation argument using contractivity of the minimizing movement yields that it is possible to uniquely define $t \to c_t$ for any $c_0$ in the $d$-closure of the set $\{G < \infty\}$. This slightly more general movement satisfies all the above mentioned properties and additionally $G(c_t) < \infty$ for any $t >0$ (for more details see \cite{ags}). 
\end{remark}

Lastly, we recall a theorem of Ba\u c\'ak, central in our later developments:

\begin{theorem} \label{BacakThm}\cite[Theorem 1.5]{ba} Given  a CAT(0) space $(M, d)$ and a $d$-lsc convex function $G : M \to (-\infty, \infty]$, assume that $G$ attains its minimum on $M$. Then any minimizing movement trajectory $t \to c_t$  weakly $d$-converges to some minimizer of $G$ as $t \to \infty$.
\end{theorem}

\section{Approximation in $d_p$ with convergent entropy}
\label{sect: approximation}
The approximation results in this section will be used in the proof of Theorem \ref{ExtKEnergyThmIntr}. Our main tools will come from the Sections 2.1 and 2.2. We begin with the simplified situation of approximation in $\Ec^1$:
\begin{lemma} \label{EntAproxLemma}Suppose $f$ is usc on $X$ with $e^{-f} \in L^1(X,\o^n)$. Given $u \in \mathcal{E}^1$, there exists $u_k \in \mathcal H_\o$ such that $d_1(u_k,u) \to 0$ and $\textup{Ent}(e^{-f}\o^n,\o_{u_k}^n) \to \textup{Ent}(e^{-f}\o^n,\o_{u}^n).$
\end{lemma}

\begin{proof} If $\textup{Ent}(e^{-f}\o^n,\o_{u}^n)=\infty$ then  any sequence $u_k \in \mathcal H_\o$ with $d_1(u_k,u) \to 0$ satisfies the requirements, as the entropy is  $d_1$-lsc. We can suppose that $\textup{Ent}(e^{-f}\o^n,\o_{u}^n) < \infty$.
Let $g = \o_{u}^n/\o^n \geq 0$ be the density function of $\o_u^n$. We will show that there exists positive functions $g_k \in C^\infty(X)$ such that $|g-g_k|_{L^1} \to 0$ and 
$$\int_M g_k \log\frac{g_k}{e^{-f}}\o^n  \to \int_M g \log\frac{g}{e^{-f}}\o^n=\textup{Ent}(e^{-f}\o^n,\o_u^n).$$
First introduce $h_k = \min \{k,g\}, \ k \in \Bbb N$. As $\phi(t) =t\log(t), \ t > 0$ is bounded from below by $-e^{-1}$ and increasing for $t > 1$, we get 
$$
-e^{-1} e^{-f} \leq h_k \log\frac{h_k}{e^{-f}} 
\leq \max \{ 0, g \log\frac{g}{e^{-f}}\}. $$ 
Clearly  $|h_k -g|_{L^1} \to 0$, and as $e^{-f}\in L^1(X,\o^n),g \log\frac{g}{e^{-f}} \in L^1(X,\o^n)$, the Lebesgue dominated  convergence theorem gives that
\begin{equation*}
\int_M h_k \log\frac{h_k}{e^{-f}}\o^n  \to \int_M g \log\frac{g}{e^{-f}}\o^n=\textup{Ent}(e^{-f}\o^n,\o_u^n).
\end{equation*}
Using the density of $C^\infty(M)$ in $L^1(M)$, by another application of the  dominated convergence theorem , we find a positive sequence $g_k \in C^\infty(X)$  such that $|g_k-h_k|_{L^1} \leq 1/k$ and 
\begin{equation*}
\Big|\int_M h_k \log\frac{h_k}{e^{-f}}\o^n  - \int_M g_k \log\frac{g_k}{e^{-f}}\o^n\Big| \leq \frac{1}{k}.
\end{equation*}
Using the Calabi-Yau theorem we find potentials $v_k \in \mathcal H_\o$ with $\sup_M v_k=0$ and $\o_{v_k}^n = g_k \o^n/\int_M g_k \o^n$. Proposition \ref{EntropyCompactnessThm} now guarantees that (after possibly passing to a subsequence) $d_1(v_k,h) \to 0$ for some $h \in \mathcal E_1(X)$. But \cite[Theorem 5 (i)]{da2} implies the equality 
of measures $\o_h^n = \o^n_u$. Finally, by  the uniqueness theorem 
\cite[Theorem B]{gz1} we get that $h$ and $u$ 
can differ by at most a constant. Hence, after possibly adding a constant, 
we can suppose that $d_1(v_k,u) \to 0$.
\end{proof}

The key point in the above proof is that a bound on the entropy implies compactness in $(\Ec^1,d_1)$. There are examples showing that the $d_2$ version of this compactness result does not hold in general. Therefore, to approximate functions in $(\Ec^p,d_p), p>1$ with convergent entropy, a new approach is necessary: 

\begin{theorem}\label{thm: Ep approximation with entropy}
Suppose $\varphi\in \Ec^p, \ p \geq 1$ and $f$ is usc on $X$ with $e^{-f}\in L^1(X,\omega^n)$.  Then there exists $\varphi_j \in \mathcal H_\o$ such that $d_p(\varphi_j,\varphi)\to 0$ as $j\to +\infty$ and $\Ent(e^{-f} \omega^n , \omega_{\varphi_j}^n) \to \Ent(e^{-f} \omega^n , \omega_{\varphi}^n)$ as $j\to +\infty$.
\end{theorem} 

\begin{proof} 
We divide the approximation procedure in three steps.
\medskip

\noindent{\bf Step 1.} Assume that $u\in \Ec^p$ has finite twisted entropy $\Ent(e^{-f}\omega^n,\omega_{u}^n)<+\infty$ and 
$$
(\omega +i\ddbar u)^n =e^{g} \omega^n,
$$
for some measurable function $g$. We also normalize $u$ so that $\int_X u (\omega+i\ddbar u)^n=0$. For each $\varepsilon>0$ let $u_{\varepsilon}\in \Ec^p(X,\omega)$ be the unique solution to 
$$
(\omega +i\ddbar u_{\vep})^n =e^{\vep u_{\vep} + g} \omega^n. 
$$
Then we claim that $d_p(u_{\vep},u)\to 0$ as $\vep \to 0$ and $\Ent(e^{-f}\omega^n,\omega_{u_{\vep}}^n)\to \Ent(e^{-f}\omega^n,\omega_{u}^n)$ as $\vep \to 0$.

Indeed, from Lemma \ref{lem: epsilon equation limit}  $u_{\vep}$ is uniformly bounded from above for $\vep\in [0,1]$, and converges in $d_p$ to $\varphi$ as $\vep\to 0$.  Also, by the  comparison principle (Lemma \ref{lem: comparison principle in E}), $u_{\vep}\geq \varphi-\sup_X \varphi$. As in the proof of Lemma \ref{EntAproxLemma} we can show using dominated convergence theorem that $\Ent(e^{-f}\omega^n,\omega_{u_{\vep}}^n)$ converges to $\Ent(e^{-f}\omega^n,\omega_{u}^n)$ as $\vep\to 0$.

\medskip

\noindent {\bf Step 2.} Let $g$ be a measurable function such that $\int_X e^{g}\omega^n <+\infty$. Assume that $u\in \Ec^p$ has finite twisted entropy $\Ent(e^{-f}\omega^n, \omega_u^n)<+\infty$ and 
$$
(\omega +i\ddbar u)^n =e^{\vep u +g}\omega^n, 
$$
for some constant $\vep>0$. Consider $g_k:=\min(g,k)$, $k\in \mathbb{N}$. Let $u_{k}\in \psh(X,\omega)\cap \Cc^0(X)$ be the unique solution to 
$$
(\omega +i \ddbar u_{k})^n = e^{\vep u_k + g_k}\omega^n.
$$
The fact that $u_k$ is continuous follows from Ko{\l}odziej's $\Cc^0$ estimate \cite{Kol98}. 
By the comparison principle $u_{k}$ is decreasing in $k$ and converges to $u$ as $k\to +\infty$. It follows from Proposition \ref{prop: monotone  implies dp convergence} that $d_p(u_k,u)\to 0$ as $k\to +\infty$. Again, the proof of Lemma \ref{EntAproxLemma} shows that $\Ent(e^{-f}\omega^n,\omega_{u_k}^n)$ converges to $\Ent(e^{-f}\omega^n,\omega_{u}^n)$ as $k\to +\infty$. 

\medskip

\noindent {\bf Step 3.} Assume that $g$ is bounded, $u\in \Ec^p$ has finite twisted entropy $\Ent(e^{-f}\omega^n, \omega_u^n)<+\infty$ and 
$$
(\omega +i\ddbar u)^n =e^{\vep u +g}\omega^n, 
$$
for some constant $\vep>0$. Let $\{g_k\}_{k\in \mathbb{N}}$ be a sequence of smooth functions converging in $L^2(X)$ to $g$. Let $u_{k}\in \mathcal H_\o$ be the unique smooth solution to 
$$
(\omega +i \ddbar u_{k})^n = e^{\vep u_k + g_k}\omega^n.
$$
The fact that $u_k$ is smooth on $X$ is well-known (see \cite{Aub}  or \cite{SzT}, \cite[Chapter 14]{GZbook} for other proofs). By Ko{\l}odziej's work \cite{Kol98} $u_k$ converges uniformly to $u$ as $k\to +\infty$, in particular $d_p(u_k,u)\to 0$. One can also check using dominated convergence theorem that the twisted entropy also converges. 

\medskip

Now, we come back to the proof of Theorem \ref{thm: Ep approximation with entropy}. If $\Ent(e^{-f}\omega^n,\omega_{\varphi}^n)=+\infty$ then any decreasing sequence $\varphi_j\in \mathcal H_\o$ which converges pointwise to $\varphi$ satisfies our requirement since the entropy is lsc with respect to weak convergence of measures. We can thus assume that $\Ent(e^{-f}\omega^n,\omega_{\varphi}^n)<+\infty$. Then we can write $\omega_{\varphi}^n =e^{g}\omega^n$. We can also assume that $\int_X \varphi \omega_{\varphi}^n=0$. Fix $\delta>0$ arbitrarily small. Denoting $\varphi_0 = \varphi$, by the three steps above we can find $\varphi_1,\varphi_2,\varphi_3 \in \Ec^p$ with $\varphi_3 \in \mathcal H_\o$, such that 
$$
d_p(\varphi_j,\varphi_{j+1})\leq \delta  \ \ \textup{ and } \ \ 
\left|\Ent\left(e^{-f}\omega^n, \omega_{\varphi_j}^n\right)-\Ent\left(e^{-f}\omega^n, \omega_{\varphi_{j+1}}^n\right)\right | \leq \delta, \ j=0,1,2.
$$
From this the result follows.
\end{proof}

\section{Extension of the twisted K-energy} \label{sect: twisted K-energy}

The main goal of this section is to prove Theorem \ref{ExtKEnergyThmIntr}. Before we can attempt a proof, we need to understand the $d_1$-continuity properties of each functional appearing in right hand side of \eqref{KEnTwistDef}. Some of the preliminary results below are well known, but as a courtesy to the reader we give a detailed account. 

\subsection{The $\AM$ functional} \label{subsect: AM functional}

The Aubin-Mabuchi functional is given by the following formula \cite[Theorem 2.3]{mab},
\begin{equation}\label{eq: AMdef}
\textup{AM}(u):=
\frac{V^{-1}}{n+1}\sum_{j=0}^{n}\int_X u\, \o^j \wedge \o_u^{n-j}, \ \ u \in \mathcal H_\o.
\end{equation}
A series of integrations by parts gives
\begin{equation}
\label{AMDifferenceEq}
\textup{AM}(v)-\textup{AM}(u)
=
\frac{V^{-1}}{n+1}\int_X(v-u)\sum_{k=0}^n \o_u^{n-k}\w \o_v^k, \ \  u,v \in \mathcal H_\o.
\end{equation}
Among other things, this formula shows that 
\begin{equation*}
u \leq v \ \Rightarrow \  \textup{AM}(u) \leq \textup{AM}(v),
\end{equation*} 
and by computing $\lim_{t \to 0}(\textup{AM}(v_t)-\textup{AM}(v))/t$ we arrive at the first order variation of $\textup{AM}$:
\begin{equation}\label{AMVarEq}
\langle D\textup{AM}(v), \delta v \rangle = V^{-1} \int_X \delta v \o^n_v, \ \  v \in \mathcal H_\o, \delta v \in C^\infty(X).
\end{equation} 
Suppose $u \in \mathcal E^1$ and let $u_j \in \mathcal H_\o$ be pointwise decreasing to $u$. Using Proposition \ref{prop: monotone implies dp convergence} we have $d_1(u,u_j)\to 0$. We hope to extend $\textup{AM}$ to $\mathcal E^1$ the following way:
\begin{equation}\label{AMExtensionFormula}
\textup{AM}(u)= \lim_j \textup{AM}(u_j).
\end{equation}
As it turns out, this extension is rigorous confirmed in the following precise result:

\begin{proposition}\label{AMextension} The map $\textup{AM} : \mathcal H_\o \to \Bbb R$ is $d_1$-Lipschitz continuous. Thus, \eqref{AMExtensionFormula} gives $d_1$-Lipschitz extension of $\textup{AM}$ to $\mathcal E^1$.
\end{proposition}
\begin{proof}First we argue that $|\textup{AM}(u_0)-\textup{AM}(u_1)| \leq d_1(u_0,u_1)$ for $u_0,u_1 \in \mathcal H_\o$. Let $[0,1] \ni t \to \gamma_t \in \mathcal H_\o$ be a smooth curve connecting $u_0,u_1$. By \eqref{AMVarEq} we can write:
$$|\textup{AM}(u_1)-\textup{AM}(u_0)| = \Big|V^{-1}\int_0^1 \int_X \dot \gamma_t \o_{\gamma_t}^n\Big|dt \leq V^{-1}\int_0^1 \int_X |\dot \gamma_t |\o_{\gamma_t}^n dt=l(\gamma).$$
Taking infimum over all smooth curves connecting $u_0,u_1$ we obtain that $|\textup{AM}(u_1)-\textup{AM}(u_0)|\leq d_1(u_0,u_1)$. The density of $\mathcal H_\o$ in $\mathcal E^1$ implies that $\textup{AM}$ extends to $\mathcal E^1$ using the formula \eqref{AMExtensionFormula}. The extension has to be $d_1$-Lipschitz continuous.
\end{proof}

Before we proceed, we mention that the ``abstract" $d_1$-continuous extension $\AM : \mathcal E^1 \to \Bbb R$ given by the above result is the same as the ``concrete" one given by the expression of \eqref{eq: AMdef} after replacing the smooth products $\o^j \wedge \o^{n-j}_u$ with the non-pluripolar products from \eqref{eq: nonpluripolar product}, as done in \cite{begz}. Moving on, we give a kind of ``domination principle" for the extended Aubin-Mabuchi energy on $\mathcal E^1$:

\begin{proposition}\label{AMdominProp} Suppose $\psi,\psi \in \mathcal E^1 $ with $\phi \geq \psi$. If $\textup{AM}(\phi) = \textup{AM}(\psi)$, then $\phi = \psi$.
\end{proposition}

\begin{proof} Suppose $\phi_k,\psi_k \in \mathcal H_\o$ are sequences pointwise decreasing to $\phi$ and $\psi$ respectively with $\phi_k \geq \psi_k$. Then \eqref{AMDifferenceEq} gives that
$$0 \leq \frac{1}{(n+1)V} \int_X (\phi_k - \psi_k) \o^n_{\psi_k} \leq \textup{AM}(\phi_k) - \textup{AM}(\psi_k)$$
Using the previous proposition and \cite[Lemma 5.2]{da2} with $\chi(t)=|t|$, $v_k = \phi_k, \ u_k = \psi_k, \ w_k = \psi_k$ we may take the limit in this estimate to obtain:
$$0 \leq \frac{1}{(n+1)V} \int_X (\phi - \psi) \o^n_{\psi} \leq \textup{AM}(\phi) - \textup{AM}(\psi)=0,$$
hence $\psi \geq \phi$ a.e. with respect to $\o^n_{\psi}$. The domination principle of the class $\mathcal E$ \cite[Propostion 5.9]{blle}  gives now that $\psi \geq \phi$ globally on $X$, hence $\psi = \phi$.
\end{proof}

The last result of this subsection points out that the family of finite energy geodesics inside $\mathcal E^p$ is in fact 'endpoint-stable'. We note that in the case $p=2$ this follows from the fact that $(\mathcal E^2,d_2)$ is CAT(0) \cite{bh}.

\begin{proposition} \label{FinEnGeodStabProp} Suppose $[0,1] \ni t \to u^j_t \in \mathcal E^p$ is a sequence of finite energy geodesic segments such that $d_p(u^j_0,u_0),d_p(u^j_1,u_1) \to 0$. Then $d_p(u^j_t,u_t) \to 0$ for all $t \in [0,1]$, where $[0,1] \ni t\to u_t \in \mathcal E^p$ is the finite energy geodesic segment connecting $u_0,u_1$.
\end{proposition}

\begin{proof} Let $t \in [0,1]$. Notice that we only have to show that any subsequence of $\{u^j_t\}_j$ contains a subsubsequence $d_p$-converging to $u_t$.

Let $\{u^{j_k}_t\}_k$ be an arbitrary subsequence of $\{u^j_t\}_j$
Let $j_{k_l}$ be the subsequence of $j_k$ satisfying the next property: there exists $v^{j_{k_l}}_i  \leq u^{j_{k_l}}_i \leq w^{j_{k_l}}_i$ with $v^{j_{k_l}}_i,w^{j_{k_l}}_i \to_{d_p} u_i$ and $\{v^{j_{k_l}}_i\}_l,\{w^{j_{k_l}}_i\}_l$ is monotone increasing/decreasing $i=0,1$. This is possible to arrange according to Proposition \ref{MonotoneDominationProp}. 
 
By $[0,1] \ni t \to v^{j_{k_l}}_t,w^{j_{k_l}}_t \in \mathcal E^p$ we denote finite energy geodesics connecting $v^{j_{k_l}}_0,v^{j_{k_l}}_1$ and $w^{j_{k_l}}_0,w^{j_{k_l}}_1$. By the maximum principle of finite energy geodesics we can write:
$$v_t := \textup{usc}\Big(\lim_l v^{j_{k_l}}_t\Big) \leq u_t \leq w_t:= \lim_l w^{j_{k_l}}_t.$$
As $\textup{AM}$ is $d_p$-continuous it follows that $\lim_l\textup{AM}(v^{j_{k_l}}_i)= \textup{AM}(u_i)=\lim_l\textup{AM}(w^{j_{k_l}}_i)$ for $i=0,1$. As $\textup{AM}$ is also linear along finite energy geodesics we get $\textup{AM}(v_t)= \textup{AM}(u_t)=\textup{AM}(w_t)$ for any $t \in [0,1]$. Proposition \ref{AMdominProp} gives that $v_t =u_t = w_t$, hence $d_p(v^{j_{k_l}}_t,u_t)\to 0$ and $d_p(w^{j_{k_l}}_t,u_t)\to 0$. Using  $v^{j_{k_l}}_t  \leq u^{j_{k_l}}_t \leq w^{j_{k_l}}_t$, \cite[Lemma 4.2]{da2} gives that $d_p(v^{j_{k_l}}_t, u^{j_{k_l}}_t) \leq d_p(v^{j_{k_l}}_t,w^{j_{k_l}}_t) \to 0$, hence $d_p(u^{j_{k_l}}_t,u_t)\to 0$ as desired at the beginning of the proof.
\end{proof}

\subsection{The $\AM_\gamma$ functional} \label{subsect: twisted AM functional}

For the moment we fix a closed (1,1)-current $\gamma$ on $X$, not necessarily postive. Recall from the introduction that the $\AM_\gamma$ is defined as follows:
\begin{equation}\label{eq: AMtwistdef}
\textup{AM}_\gamma(u):=
\frac{1}{nV}\sum_{j=0}^{n-1}\int_X u\, \gamma \wedge \o^j \wedge \o_u^{n-1-j}, u \in \mathcal H_\o.
\end{equation}
Similarly to $\AM$, integrating by parts gives
\begin{equation}
\label{AMtwistDifferenceEq}
\textup{AM}_\gamma(v)-\textup{AM}_\gamma(u)
=
\frac{1}{nV}\int_X(v-u)\sum_{k=0}^{n-1} \gamma \wedge \o_u^{n-k-1}\w \o_v^k.
\end{equation}
When $\gamma \geq 0$ this last formula gives: 
\begin{equation*}
u \leq v \ \Rightarrow \  \textup{AM}_\gamma(u) \leq \textup{AM}_\gamma(v).
\end{equation*} 
By computing $\lim_{t \to 0}(\textup{AM}_\gamma(v_t)-\textup{AM}_\gamma(v))/t$ we arrive at the first order variation of $\textup{AM}_\gamma$:
\begin{equation}\label{AMTwistVarEq}
\langle D\textup{AM}_\gamma(v), \delta v \rangle = V^{-1} \int_X \delta v \gamma \wedge \o^{n-1}_v, \ \  v \in \mathcal H_\o, \delta v \in C^\infty(X).
\end{equation} 
\paragraph{Extension of $\textup{AM}_\gamma$ to $\mathcal E^1$ when $\gamma$ is smooth.}
For this paragraph suppose $\gamma$ is smooth. Suppose $u \in \mathcal E^1$ and let $u_j \in \mathcal H_\o$ be pointwise decreasing to $u$. Using Proposition \ref{prop: monotone implies dp convergence} we have $d_1(u,u_j)\to 0$. We hope to extend $\textup{AM}_\gamma$ to $\mathcal E^1$ the following way:
\begin{equation}\label{AMTwistExtensionFormula}
\textup{AM}_\gamma(u)= \lim_j \textup{AM}_\gamma(u_j).
\end{equation}
As it turns out, this extension is rigorous as we have the following precise result:

\begin{proposition}\label{AMtwistextension} Formula \eqref{AMTwistExtensionFormula} gives a $d_1$-continuous functional $\textup{AM}_\gamma : \mathcal E^1 \to \Bbb R$. Additionally, $\textup{AM}_\gamma$ thus extended is bounded on $d_1$-bounded subsets of $\mathcal E_1$.
\end{proposition}

\begin{proof} We argue that for any $R>0$ there exists  $f_R:\Bbb R \to \Bbb R$ continuous with $f_R(0)=0$ such that
\begin{equation} \label{PropMainEst}
|\textup{AM}_\gamma(u_0)-\textup{AM}_\gamma(u_1)| \leq f_R(d_1(u_0,u_1)),
\end{equation}
for any $u_0,u_1 \in \mathcal H_\o\cap\{v\,:\,d_1(0,v)\le R\}$.  We have $-C\o \leq \gamma \leq C\o$ for some $C>1$. Using \eqref{AMtwistDifferenceEq} and  the observation
$\o_{(u_0+u_1)/{4}}=\o/2+\o_{u_0}/4+\o_{u_1}/4$ it follows that
$$|\textup{AM}_\gamma(u_0)-\textup{AM}_\gamma(u_1)| \leq C \int_X|u_0-u_1| \o^n_{(u_0 + u_1)/4}$$
By \cite[Corollary 5.7]{da2} and its proof, for each $R>0$ there exists a continuous function 
$f_R:\Bbb R \to \Bbb R$ with $f_R(0)=0$ such that
\begin{equation*}
\int_X |v-w|\o_h^n \leq f_R(d_1(v,w)),
\end{equation*}
for any $v,w,h \in \mathcal E_1\cap\{v\,:\,d_1(0,v)\le R\}$.
Using this last fact, to argue that \eqref{PropMainEst} holds, it is enough to show that $d_1(0,(u_0+u_1)/4)$ is bounded in terms of $d_1(0,u_0)$ and $d_1(0,u_1)$. We recall  \cite[Lemma 5.3]{da2} that says that there exists $D>1$ such that $d_1(a,(a+b)/2) \leq D d_1(a,b)$ for any $a,b \in \mathcal E_1$. Using this several times along with the triangle inequality, we can write
\begin{flalign*}
d_1(0 ,(u_0+u_1)/4) \leq & Cd_1(0,(u_0+u_1)/2) \leq C(d_1(0,u_0)  + d_1(u_0,(u_0+u_1)/2)) \\
\leq & C^2(d_1(0,u_0) + d_1(u_0,u_1)) \leq 2C^2(d_1(0,u_0) + d_1(0,u_1)),
\end{flalign*}
finishing the proof.
\end{proof}

As in the case of $\AM$,  we mention that the ``abstract" $d_1$-continuous extension $\AM_\gamma : \mathcal E^1 \to \Bbb R$ given by the above result is identical to the one given by the ``concrete" expression of \eqref{eq: AMtwistdef} after replacing the smooth products $\gamma \wedge \o^j \wedge \o^{n-j-1}_u$ with non-pluripolar products similar to \eqref{eq: nonpluripolar product}.

\paragraph{Convexity and extension of $\textup{AM}_\chi$ to $\mathcal H^\Delta_\o$ when $\chi$ satisfies \eqref{eq: ChiProp}.} Suppose $\chi=\beta +i\ddbar f$ is a (1,1)-current satisfying \eqref{eq: ChiProp}. Observe that it is not possible to extend $\textup{AM}_\chi$ to $\mathcal H^\Delta_\o$ using the techniques of the previous paragraph directly. Instead, using integration by parts, we notice that given $u \in \mathcal H_\o$ we have an alternative formula for $\textup{AM}_\chi(u)$:
\begin{flalign}
\textup{AM}_\chi(u) &= 
\frac{1}{nV}\sum_{j=0}^{n-1}\int_X u\, \beta \wedge \o^j \wedge \o_u^{n-1-j} + \frac{1}{nV}\int_X f (\o^n_u - \o^n) \nonumber\\
&=\textup{AM}_\beta(u) + \frac{1}{nV}\int_X f (\o^n_u - \o^n).\label{AMtwistAltDef}
\end{flalign}
As $\beta$ is smooth, $\textup{AM}_\beta$ extends $d_1$-continuously to $\mathcal H^\Delta_\o$ by the previous paragraph. The map $u \to \int_X f \o^n_u$ clearly makes sense and is finite for all $u \in \mathcal H^\Delta_\o$, hence using \eqref{AMtwistAltDef} it is possible to extend $\textup{AM}_\chi$ to $\mathcal H^\Delta_\o$. Though not needed, it can be further shown that this  extension is independent of the choice of $\beta$ and $f$. 
 
Given $u_0,u_1 \in \mathcal H_\o$, for the weak geodesic $[0,1] \ni t \to u_t \in \mathcal H^\Delta_\o$ connecting $u_0,u_1$ we would like to show that $t \to \textup{AM}_\chi(t)$ is convex. When $\chi$ is smooth this follows from the results of \cite{bb} and \cite{c2}. It turns out that for more general $\chi$ the same proof gives an analogous result:

\begin{proposition}\label{AMtwistConvexity}
Suppose $\chi=\beta +i\ddbar f \geq 0$ satisfies \eqref{eq: ChiProp}. Equation \eqref{AMtwistAltDef} gives an extension $\textup{AM}_\chi: \mathcal H^\Delta_\o \to \Bbb R$ for which $t \to\textup{AM}_\chi(u_t)$ is convex for any weak geodesic segment $[0,1] \ni t \to u_t \in \mathcal H^\Delta_\o$.
\end{proposition}

\begin{proof}Suppose $t_1 \geq t_0$. When $\chi$ is smooth, it is well known that for $[0,1] \ni t\to v_t \in \mathcal H_\o$ smooth subgeodesic (i.e. $\pi^* \o + i\ddbar v \geq 0$) we actually have
$$\frac{d}{dt}\big|_{t=t_1}\textup{AM}_\chi(v_t) - \frac{d}{dt}\big|_{t=t_0}\textup{AM}_\chi(v_t)= \int_{S_{t_0,t_1} \times X} \pi^*\chi \wedge (\pi^*\o + i\ddbar v)^{n}, $$
where $S_{t_0,t_1} \subset \Bbb C$ is the strip $\{ t_0 \leq \textup{Re }z \leq t_1\}.$ Hence, $t \to \textup{AM}_\chi(v_t)$ is convex. We claim that the same proof goes through for any positive closed current $\chi=\beta + i\ddbar f$ as well.

 When dealing with a weak geodesic $[0,1] \ni t \to u_t \in \mathcal H^\Delta_\o$, it is possible to approximate it uniformly with a decreasing sequence of smooth subgeodesics $t \to u^\varepsilon_t$ called $\varepsilon$-geodesics (see \cite{c}). As all measures $\o^n_{u^\varepsilon_t}$ have uniformly bounded density and converge weakly to $\o^n_{u_t}$ we have
 $$\lim_{\varepsilon \to 0}\int_X f \o^n_{u^\varepsilon_t}=\int_X f \o^n_{u_t} \ \textup{ and } \ \lim_{\varepsilon \to 0}\textup{AM}_\beta(u^\varepsilon_t)=\textup{AM}_\beta(u_t).$$
Hence, after taking limit in \eqref{AMtwistAltDef} it follows that $t \to \lim_{\varepsilon \to 0} AM_\chi(u^\varepsilon_t)=\textup{AM}_\chi(u_t)$ is convex.
\end{proof}
Finally, we note the following useful inequality for $\AM_\gamma$.
\begin{lemma}\label{lem: bounds for AM}
Let $\psi\in \Ec^1$ and set $\theta=\omega_\psi$. For any $u,v\in \Ec^1$ we have
	$$
	\frac{1}{V}\int_X (u-v) \omega_u^{n-1}\wedge \theta \leq \AM_{\theta}(u)-\AM_{\theta}(v) \leq \frac{1}{V}\int_X (u-v) \omega_v^{n-1}\wedge \theta.
	$$
For $\AM$ we have similar inequalities
	$$
	\frac{1}{V}\int_X (u-v) \omega_u^{n}\leq \AM(u)-\AM(v) \leq \frac{1}{V}\int_X (u-v) \omega_v^{n}.
	$$
\end{lemma}
\begin{proof}
	Using (\ref{AMDifferenceEq}) and (\ref{AMtwistDifferenceEq}) the desired inequalities simply follow from the fact that 
	$$
	\int_X (u-v) i\ddbar (u-v) \wedge T\leq 0,
	$$
	for any $T=\omega_{\varphi_1}\wedge \ldots \wedge \omega_{\varphi_{n-1}}$ with $\varphi_j\in \Ec^1, \forall j$.
\end{proof}

\subsection{The twisted K-energy} \label{sect: the twisted K-energy}

For the remainder of the paper suppose $\chi = \beta + i\ddbar f$ satisfies \eqref{eq: ChiProp} unless specified otherwise.
Recall that the twisted K-energy $\mathcal K_\chi: \mathcal H_\o \to \Bbb R$ is defined as follows:
\begin{flalign*}
\mathcal K_\chi &=\textup{Ent}(e^{-f}\o^n,\o^n_u)
+ \overline{S}_\chi \textup{AM}(u) - n\textup{AM}_{\Ric \o-\beta}(u)-\int_X f \o^n.
\end{flalign*}
When $f$ is smooth, recall the following formula for the variation of the entropy:
$$\langle D\textup{Ent}(e^{-f}\o^n,\o^n_v),\delta v\rangle = nV^{-1}\int_X \delta v  (\Ric \o - \Ric \o_v + i\ddbar f) \wedge \o_v^{n-1}.$$
When $\chi$ is smooth, putting the above formula, \eqref{AMVarEq} and \eqref{AMTwistVarEq} together we obtain:
\begin{flalign}\nonumber 
\langle D\mathcal K_\chi(v),\delta v\rangle &= \frac{n}{V}\int_X \delta v  (\overline{S}_\chi\o_v - \Ric \o_v  + \chi) \wedge \o_v^{n-1}\nonumber \\
&=V^{-1}\int_X \delta v (\overline{S}_\chi - S_{\o_v}  + \textup{Tr}^{\o_u}\chi) \o_v^{n}.\nonumber
\end{flalign}
We arrive at the main theorem of this section:

\begin{theorem} \label{ExtKEnergyThm}  Suppose $(X,\o)$ is a compact connected K\"ahler manifold and $\chi=\beta + i\ddbar f$ satisfies \eqref{eq: ChiProp}. The twisted K-energy can be extended to a functional $\mathcal K_\chi:\mathcal E^1 \to \Bbb R \cup \{ \infty\}$ using the formula:
\begin{equation}\label{KenExtThmFormula}
\mathcal K_\chi(u) = \textup{Ent}(e^{-f}\o^n,\o^n_u)
+ \overline{S}_\chi \textup{AM}(u) - n\textup{AM}_{\Ric \o-\beta}(u) - \int_X f\o^n.
\end{equation}
Thus extended, $\mathcal K_\chi|_{\mathcal E^p}$ is the greatest $d_p$-lsc extension of $\mathcal K_\chi |_{\mathcal H_\o}$ for any $p\geq 1$. Additionally, $\mathcal K_\chi|_{\mathcal E^p}$ is convex along the finite energy geodesics of $\mathcal E^p$.
\end{theorem}

\begin{proof}First we argue that the expression given by \eqref{KenExtThmFormula} does give a $d_1$-lsc function on $\mathcal E^1$. Indeed, by Propositions \ref{AMextension} and \ref{AMtwistextension} the functionals $\textup{AM}$ and $\textup{AM}_{\Ric \o -\beta}$ admit a $d_1$-continuous extension to $\mathcal E^1$. Lastly, as $d_1$-convergence of a potentials implies weak convergence of the corresponding complex Monge-Amp\`ere measures, it follows that the correspondence $u \to  \textup{Ent}(e^{-f}\o^n,\o^n_u)$ is $d_1$-lsc. When restricted to $\mathcal E^p$, \eqref{KenExtThmFormula} is additionally $d_p$-lsc, because $d_p$-convergence dominates $d_1$-convergence for any $p>1$.

We now show that thus extended $\mathcal K_\chi|_{\mathcal E^p}$  is indeed the greatest $d_p$-lsc extension of $\mathcal K_\chi|_{\mathcal H_\o}$. For this we only have to argue that for any $u \in \mathcal E^p$ there exists $u_j \in \mathcal H_\o$ such that $d_p(u_j,u) \to 0$ and 
$$
\mathcal K_\chi(u) = \lim_j \mathcal K_\chi(u_j).
$$
As $\textup{AM}(\cdot)$ and $\textup{AM}_{\Ric -\beta}(\cdot)$ is $d_p$-continuous, this is exactly the content of Theorem \ref{thm: Ep approximation with entropy}.

Since finite energy geodesics of $\Ec^p$ are also finite energy geodesics in $\Ec^1$, it remains  to show that for any finite energy geodesic $[0,1] \ni t \to u_t \in \mathcal E^1$ the curve $t \to \mathcal K_\chi(u_t)$ is convex and continuous. 

Suppose $t_0,t_1 \in [0,1]$ with $t_0 \leq t_1$. As $\mathcal K_\chi$ was extended in the greatest $d_1$-lsc manner, we can find 
$u^k_{t_0},u^k_{t_0} \in \mathcal H_\o$ with $d_1(u^k_{t_0},u_{t_0})\to 0$, $d_1(u^k_{t_1},u_{t_1})\to 0$ and
$$\mathcal K_\chi(u_{t_0})= \lim_k \mathcal K_\chi(u^k_{t_0}), \ \ \mathcal K_\chi(u_{t_1})= \lim_k \mathcal K_\chi(u^k_{t_1}).$$

Let $[t_0,t_1] \ni t \to u^k_t \in \mathcal H^\Delta_\o$ be the weak geodesics connecting $u^k_{t_0},u^k_{t_1}$. By Proposition \ref{FinEnGeodStabProp} we get that $d_1(u^k_{t},u_t) \to 0$ for any $t \in [t_0,t_1]$. 
Note that for $u \in \mathcal H^\Delta_\o$ we can write:
$$\mathcal K_\chi(u) = \textup{Ent}(\o^n,\o^n_u)
+ \bar{S}_\chi\textup{AM}(u) - n\textup{AM}_{\Ric \o}(u)+\Big(n\textup{AM}_{\beta}(u) + \frac{1}{V}\int_X f\o_u^n\Big).$$
Using this, Proposition \ref{AMtwistConvexity}, \cite[Theorem 1.1]{bb} and the linearity of $\textup{AM}$ along finite energy geodesics, it follows that $t \to \mathcal K_\chi(u^k_t)$ is convex on $[0,1]$. As $\mathcal K_\chi:\mathcal E^1 \to \Bbb R \cup \{\infty\}$ is $d_1$-lsc it follows that
\begin{flalign*}
\mathcal K_\chi(u_t) \leq \liminf_k \mathcal K_\chi(u^k_t) &\leq \frac{t-t_0}{t_1 - t_0} \lim_k \mathcal K_\chi(u^k_{t_0}) + \frac{t_1-t}{t_1 - t_0} \lim_k \mathcal K_\chi(u^k_{t_1})\\
& \leq \frac{t-t_0}{t_1 - t_0} \mathcal K_\chi(u_{t_0}) + \frac{t_1-t}{t_1 - t_0} \mathcal K_\chi(u_{t_1}),
\end{flalign*}
hence $[0,1] \ni t \to \mathcal K_\chi(u_t) \in (-\infty,\infty]$ is convex. As $\mathcal K_\chi$ is $d_1$-lsc it follows additionally that $t \to \mathcal K_\chi(u_t)$ is continuous up to the boundary of $[0,1]$.
\end{proof}

Finally, we bring Theorem \ref{EntropyCompactnessThm} in a form that will be most convenient to use in our later developments:
\begin{corollary} \label{KChiCompactnessCor} Suppose $\chi = \beta + i\ddbar f$ satisfies \eqref{eq: ChiProp} and $\{u_k\}_k \subset \mathcal E^1$ is a sequence for which the following holds:
$$
d_1(0,u_k) < C, \  \mathcal K_\chi({u_k}) < C. 
$$
Then $\{ u_k\}_k$ contains a $d_1$-convergent subsequence. 
\end{corollary}
\begin{proof} By \eqref{dpCharFormula} it follows that $|\sup_X u_k |< C$. From \eqref{KenExtThmFormula} and Propositions \ref{AMextension} and \ref{AMtwistextension} we get that $\textup{Ent}(e^{-f}\o^n,\o_{u_k}^n)$ is also uniformly bounded. Now we can invoke Theorem \ref{EntropyCompactnessThm} to finish the argument.
\end{proof}

\subsection{Convexity
in the finite entropy space.}\label{Sec: finite entropy section}
Suppose $\chi=\beta + i\ddbar f$ satisfies \eqref{eq: ChiProp}. Denote by $\mbox{Ent}_\chi{(X,\o)}$ the space of finite entropy potentials:
$$
\mbox{Ent}_\chi{(X,\o)}= \{ u \in \mathcal E(X,\o), \ \ \mbox{Ent}(e^{-f}\o^n,\o_u^n) < \infty\}.
$$
Observe that $\mbox{Ent}_\chi(X,\o)$ is independent of the choice of $\beta$ and $f$. Also, as we show that $\mbox{Ent}_\chi(X,\o)$ is contained in the finite energy space $\mathcal{E}^{1}$:
\begin{lemma}\label{FinEntFinEnLemma} Suppose $\chi=\beta + i\ddbar f$ satisfies \eqref{eq: ChiProp}. Then $\textup{Ent}_\chi{(X,\o)} \subset \mathcal E^1$.
\end{lemma}
\begin{proof}Suppose $u \in \textup{Ent}_\chi{(X,\o)}$ with $\o_u^n = h \o^n$. Note that  the functions $\phi,\psi:[0,\infty) \to [0,\infty)$ given by $\phi(t)=(t+1)\log(t+1)-t$ and $\psi(t)=e^t -t -1$ are convex conjugates of each other, implying that $ab \leq \phi(a) + \psi(b)$. Using this we can write:
\begin{flalign*}
\int_X |u|\o^n_u =&\int_X |u| (he^f)e^{-f} \o^n\\
\leq & \int_X (e^{|u|} -|u|-1)e^{-f}\o^n +  \int_X ((he^{f}+1)\log(he^{f}+1)-h)e^{-f}\o^n.
\end{flalign*}
To finish the proof it is enough to argue that both terms in this last expression are bounded. For the first term, suppose $1/p + 1/q=1$. Using Young's inequality we arrive at
$$\int_X (e^{|u|} -|u|-1)e^{-f}\o^n \leq \frac{1}{q}\int_X (e^{|u|} - |u|-1)^q \o^n + \frac{1}{p} \int_X e^{-pf} \o^n.$$
As $u$ has zero Lelong numbers \cite[Corollary 1.8]{gz1}, the first integral is finite as follows from Skoda's theorem. For an appropriate $p$ the second integral is bounded, as $e^{-f} \in L^p(X,\o^n)$ for some $p > 1$.

For the second term, observe that for $t$ big enough $\phi(t) \leq 2 t \log t$, hence we can write:
\begin{flalign*}
\int_X ((he^{f}+1)\log(he^{f}+1)-h)e^{-f}\o^n &\leq 2 \int_X h\log(he^f)\o^n+C\\
&=2V\textup{Ent}(e^{-f}\o^n,\o_u^n) + C.
\end{flalign*}
\end{proof}
As a consequence of Theorem \ref{ExtKEnergyThm} we obtain that  $\textup{Ent}_\chi(X,\o) \subset \mathcal E^1$ is to some extent  ``geodesically convex'': 

\begin{theorem}
\label{FinEntGeodConvThm}
Suppose $\chi =\beta + i\ddbar f$ satisfies \eqref{eq: ChiProp}. Then $(\textup{Ent}_\chi(X,\o),d_1)$ is a geodesic sub-metric space of $(\mathcal E^1(X,\o),d_1)$. Additionally, if $\Ric \o \geq \beta$ then the map $\textup{Ent}_\chi(X,\o) \ni u \to \textup{Ent}(e^{-f}\o^n,\o_u^n) \in \Bbb R$  is convex along finite energy geodesics.
\end{theorem}

\begin{proof} Suppose $u_0,u_1 \in \textup{Ent}_\chi(X,\o)$. Let $[0,1] \ni t \to u_t \in \mathcal E^1$ be the finite energy geodesic connecting $u_0,u_1$. By Theorem \ref{ExtKEnergyThm} it follows that $t \to \mathcal K_\chi(u_t)$ is convex on $[0,1]$ hence $\mathcal K_\chi(u_t)$ is finite for all $t \in [0,1]$. Using the finiteness of $\textup{AM}$ and $\textup{AM}_{\Ric \o -  \beta}$, this necessarily gives that $\textup{Ent}(e^{-f}\o^n,\o_{u_t}^n)$ is also finite for all $t \in [0,1]$. 

For the last statement notice that $t \to n \textup{AM}_{\Ric \o -\beta}(u_t) - \overline{S}_\chi \textup{AM}(u_t)$ is convex, as follows from Proposition \ref{AMtwistConvexity}. As $t \to \mathcal K_\chi(u_t)$ is also convex, from \eqref{KenExtThmFormula} it follows that $t \to \textup{Ent}(e^{-f}\o^n,\o_{u_t}^n)$ is also convex.
\end{proof}
In the case $\beta =0$, this convexity result can be seen as the complex version of one of the central results of the theory of optimal transport of measure, which says that, if $g_0$ is a given Riemannian metric on a compact real manifold $X$ with non-negative Ricci curvature and whose normalized volume form is denoted by $\mu_0$, then the relative entropy function $\mu \to \textup{Ent}(\mu_0,\mu)$ is convex along curves $t \to \mu_t$ defined by McCann's displacement interpolation (which may be formulated in terms of optimal transport maps). The latter curves can be seen as weak geodesics for Otto's Riemannian metric on the
space of all normalized volume forms on X. More precisely, the curves $t \to \mu_t$ are the geodesics in the metric space $(\mathcal P(M),d_{W_2})$ defined by the space $\mathcal P(M)$ of all probability measures on $X$ equipped with the  Wasserstein 2-metric, which can be viewed
as a completion of Otto's Riemannian structure \cite{vil}. Hence, the role of Otto's Riemmanian metric is on the present complex
setting played by Mabuchi's Riemannian metric.

\subsection{Uniqueness of twisted K-energy minimizers}
In this subsection we suppose $\chi$ is a K\"ahler form. We are going to prove that there is at most one minimizer in $\Ec^1$ of the twisted $K$-energy $\Kcc$. We need the following result which may be of independent interest.
\begin{lemma}\label{lem: derivative of AM along geodesic}
	Let $\varphi_0,\varphi_1\in \Ec^1$ and $[0,1]\ni t\to\varphi_t$ be the finite energy geodesic connecting  $\varphi_0$ and $\varphi_t$. Suppose that $\omega_{\varphi_t}^n$ is absolutely continuous with respect to $\omega^n$ for every $t\in [0,1]$. Then for almost every $t\in (0,1)$ we have
	\begin{equation}
		\label{eq: derivative of geodesic}
		\AM(\varphi_1)-\AM(\varphi_0)=\frac{1}{V}\int_X \dot{\varphi}^+_t \omega_{\varphi_t}^n=\frac{1}{V}\int_X \dot{\varphi}^-_t  \omega_{\varphi_t}^n,
	\end{equation}
	where, for fixed $x\in X$, $\dot{\varphi}^+_t(x)$ and $\dot{\varphi}^-_t(x)$ are the right  and left derivative of $t\to \varphi(t,x)$ respectively. 
\end{lemma} 
\begin{proof}
For simplicity we assume that $V=1$.
Fix two real numbers $a,b$ such that $0<a<b<1$. We first observe that for $t\in (a,b)$ and $h>0$ small enough, by convexity we have
$$
\frac{\varphi_t-\varphi_0}{t} \leq \frac{\varphi_{t+h}-\varphi_t}{h}\leq \frac{\varphi_1-\varphi_t}{1-t}. 
$$ 
It thus follows that both $\dot{\varphi}^+_t$ and $\dot{\varphi}^-_t$ are integrable with respect to $\omega_{\varphi_t}^n$. 
From Lemma \ref{lem: bounds for AM} we obtain 
$$
\AM(\varphi_{t+h})-\AM(\varphi_t) \leq \int_X (\varphi_{t+h}-\varphi_t) \omega_{\varphi_t}^n. 
$$
Since $\AM$ is linear along the weak geodesic $\varphi_t$, by dividing the above inequality by $h$ and letting $h\to 0$ we obtain
\begin{equation}
		\label{eq: derivative of geodesic 1}
		\int_X \dot{\varphi}^-_t  \omega_{\varphi_t}^n \leq \AM(\varphi_1)-\AM(\varphi_0)\leq \int_X \dot{\varphi}^+_t \omega_{\varphi_t}^n.
\end{equation}
For each $x\in X$ the function $t\to \varphi_t(x)$ is convex, hence differentiable almost every where in $[0,1]$. It follows that the set
$$
\{(x,t)\in X\times [a,b] : \dot{\varphi}^-_t(x)<\dot{\varphi}^+_t(x)\}
$$
has zero measure (where the measure here is the product of  $\omega^n$ and $dt$). Let $f(t,x)$ be the density of the Monge-Amp\`ere measure   $(\omega+i\ddbar \varphi_t)^n$. We then have 
\begin{equation}
		\label{eq: derivative of geodesic 2}
		\int_{X\times [a,b]} \dot{\varphi}^-_t f(t,x)\omega^n dt = \int_{X\times[a,b]} \dot{\varphi}^+_t f(t,x)\omega^n dt.
\end{equation}
Now, by Fubini's theorem, (\ref{eq: derivative of geodesic 1}) and (\ref{eq: derivative of geodesic 2}) we see that the inequalities in (\ref{eq: derivative of geodesic 1}) become equalities for almost every $t$ in $[a,b]$, completing the proof.
\end{proof}

\begin{theorem}\label{thm: linear along geodesic}
Let $\alpha$ be a K\"ahler form. Let $\varphi_0,\varphi_1\in \Ec^1$ and $\varphi_t$ be the finite energy geodesic connecting $\varphi_0$ and $\varphi_1$. Suppose that $\omega_{\varphi_t}^n$ is subordinate to $\omega^n$ for any $t\in [0,1]$. If  $\AM_{\alpha}$ is linear along $\varphi_t$ then $\varphi_1-\varphi_0$ is constant. 	
\end{theorem}
\begin{proof}
	We can assume that $\AM(\varphi_0)=\AM(\varphi_1)$ and we normalize $\omega$ so that $V=1$.  We claim that $\AM_{\beta}$ is also linear along $\varphi_t$, where $\beta$ is any K\"ahler form. Indeed, multiplying $\beta$ by some small positive constant we can assume that $\gamma:=\alpha-\beta>0$. It follows from Proposition \ref{AMtwistConvexity} that both $t \to \AM_{\gamma}(\varphi_t)$ and $t \to \AM_{\beta}(\varphi_t)$ are convex.  Because $\AM_{\alpha}=\AM_{\beta}+\AM_{\gamma}$ is linear along $\varphi_t$, it follows that in fact $t \to \AM_{\beta}(\varphi_t)$ is linear as well. 
By approximation it follows that $\AM_{\omega_{\psi}}$ is linear along $\varphi_t$ for any $\psi\in \Ec^1$. 
	
Fix $t \in (0,1)$ such that \eqref{eq: derivative of geodesic} holds in Lemma \ref{lem: derivative of AM along geodesic}. For $h>0$ small enough we have
\begin{eqnarray*}
	\int_X \frac{\varphi_{t+h}-\varphi_t}{h}\omega_{\varphi_{t}}^n &\geq & \frac{\AM_{\omega_{\varphi_t}}(\varphi_{t+h}) - \AM_{\omega_{\varphi_t}}(\varphi_{t})}{h} \\
	&=& -\frac{1}{n} \int_X \frac{\varphi_{t+h}-\varphi_t}{h}\omega_{\varphi_{t+h}}^n \\ 
	&\geq & -\frac{1}{n} \int_X \frac{\varphi_{t+h}-\varphi_t}{h}\omega_{\varphi_{t}}^n.
\end{eqnarray*}
	In the first line we have used Lemma \ref{lem: bounds for AM}. In the second line we have used the assumption that $\AM$ is constant along $\varphi_t$. In the last line we have used again Lemma \ref{lem: bounds for AM}. Now, letting $h \to 0$ and using Lemma \ref{lem: derivative of AM along geodesic} we see that the right derivative of $l \to \AM_{\o_{\varphi_t}}(\varphi_l)$ at $t$ is zero. Thus $l \to \AM_{\o_{\varphi_t}}(\varphi_l)$ is in fact constant. This combined  with $l \to \AM(\varphi_l)$ being constant imply that
\begin{equation}\label{eq: integral equals zero}
0=(n+1)(\AM(\varphi_1)-\AM(\varphi_t))-n(\AM_{\o_{\varphi_t}}(\varphi_1)-\AM_{\o_{\varphi_t}}(\varphi_t))=\int_X (\varphi_1-\varphi_t) \omega_{\varphi_1}^n.
\end{equation}
A computation similar to the one in Lemma \ref{lem: bounds for AM} gives that all terms in the expression of $\AM(\varphi_1)-AM(\varphi_t)$ from \eqref{AMDifferenceEq} are greater than
$\int_X (\varphi_1-\varphi_t) \omega_{\varphi_1}^n$. Using this, \eqref{eq: integral equals zero}  and $\AM(\varphi_1)-AM(\varphi_t) =0$ we obtain $\int_X (\varphi_1-\varphi_t) \omega_{\varphi_t}^n =0$. Together with \eqref{eq: integral equals zero} this gives 
$$I(\varphi_1,\varphi_t)=\int_X (\varphi_1-\varphi_t)(\o^n_{\varphi_t}-\o^n_{\varphi_1})=0.$$ 

Hence, by the results in \cite[Section 2.1]{bbegz}, the difference $\varphi_t-\varphi_1$ is constant. In fact $\varphi_t = \varphi_1$, as the Aubin-Mabuchi energy is constant along the geodesic $l \to \varphi_l$. This implies $\varphi_s=\varphi_1, \forall s\in [0,1]$, as $(\Ec^1,d_1)$ is a geodesic space.
\end{proof}
We are now ready to prove the uniqueness result.
\begin{theorem}\label{thm: uniqueness of twisted minimizer}
	Assume that $\chi$ is a K\"ahler form. If $\varphi_0$ and $\varphi_1$ are minimizers in $\Ec^1$ of the twisted Mabuchi energy $\Kcc$ then
	$\varphi_1-\varphi_0$ is constant.
\end{theorem}
\begin{proof}Let $t \to \varphi_t$ be the finite energy geodesic connecting $\varphi_0,\varphi_1$. By the convexity of $\Kcc$ it follows that $\Kcc$ is linear along $t \to \varphi_t$. Since $t \to \AM(\varphi_t), \Kcc(\varphi_t)$ are linear and $t \to \AM_{\chi}(\varphi_t), \mathcal K(\varphi_t)$ are convex, the decomposition $\Kcc=\Kc + (\bar{S}_{\chi}-\bar{S})\AM +  n\AM_{\chi}$ then reveals  that $\AM_{\chi}$ is also linear along $t \to \varphi_t$ and $\o_{\varphi_t}^n$ is subordinate to $\o^n$. The result now follows from Theorem \ref{thm: linear along geodesic}.
\end{proof}

\begin{remark}\label{remark: DRremark} When $\chi$ is a K\"ahler form, using this last theorem, it  can be seen that the conditions (A1)-(A4) and (P1)-(P7) are verified in \cite[Theorem 3.4]{dr} for the data $(\mathcal E^1, d_1, \mathcal K_\chi, \{ Id \})$ to give that a minimizer of $\mathcal K_\chi$ exists in $\mathcal E^1$ if and only if there exists $C,D>0$ such that
$$\mathcal K_\chi(u) \geq C d_1(0,u) -D, \ \ u \in \mathcal H_\o.$$
This verifies a weak version of \cite[Conjecture 1.21]{c4} going back to \cite[Conjecture 6.1]{c3}. For related partial results, see also \cite{de}.
\end{remark}

\section{Relating $d_1$-convergence to weak $d_2$-convergence}
\label{sec: weak d2 convergence}
Before we get into the details of our particular situation, we start with a pedagogical example: suppose $(M,\mu)$ is a measure space with finite volume. By $(L^p(M,\mu),\| \cdot \|_p)$ we denote the usual $L^p$ spaces on $M$. From H\"older's inequality it follows that on $L^2(M,\mu)$ the $\| \cdot \|_2$ norm dominates the $\| \cdot \|_1$ norm. Our focus however is on the weak-$L^2$ topology. As it turns, the $L^1$-topology dominates the weak-$L^2$ topology. The simple explanation for this is that $L^1$-balls inside $L^2(M,\mu)$ are closed convex sets, and it is a classical fact that weak $L^2$-limits do not exit closed convex sets. 
Though much simplified, as it turns out, these idea generalizes to the setting of the metric spaces $(\mathcal E^p,d_p)$. As we show below, the $d_1$-metric balls have a certain convexity property that will make these sets $d_2$-convex and closed inside $\mathcal E^2$. This will imply that $d_1$-convergence dominates weak $d_2$-convergence.  In the next section, coupled with Theorem \ref{BacakThm}, this fact will have implications on the convergence of the weak twisted Calabi flow.

As advocated in \cite{da1,da2}, a proper understanding of the 'rooftop' envelopes $P(u_0,u_1)$ gives insight into the geometry of the spaces $(\mathcal E^p,d_p)$. Furthering this relationship, we state the following proposition:

\begin{proposition} Suppose $[0,1] \ni t \to u_t,v_t \in \mathcal E^1$ are finite energy geodesics. Then the map $t \to \textup{AM}(P(u_t,v_t))$ is concave. Consequently, the map $t \to d_1(u_t,v_t)$ is convex. 
\end{proposition}
\begin{proof} Let $a,b \in [0,1]$. As shown in \cite[Theorem 3]{da1} it follows that $P(u_a,v_a),P(u_b,v_b) \in \mathcal E^1(X,\o)$. Let $[0,1] \ni t \to w_t \in \mathcal E^1(X,\o)$ be a finite energy geodesic connecting $w_0=P(u_a,v_a)$ and $w_1=P(u_b,v_b)$. By the maximimum principle of finite energy geodesics we have $w_t \leq u_{ta + (1-t)b},v_{ta + (1-t)b}$, hence also $w_t \leq P(u_{ta + (1-t)b},v_{ta + (1-t)b})$. By the monotonicity of the Aubin-Mabuchi energy and since $t \to \textup{AM}(w_t)$ is linear we obtain
$$t \textup{AM}(P(u_a,v_a)) + (1-t)\textup{AM}(P(u_b,v_b))=\textup{AM}(w_t) \leq \textup{AM}(P(u_{ta + (1-t)b},v_{ta + (1-t)b})).$$
The last statement of the proposition follows from the linearity of $\textup{AM}$ along finite energy geodesics, the concavity we just established and the explicit formula for $d_1$ given in \cite[Corollary 4.14]{da2}, according to which $d_1(u_t,v_t)=\textup{AM}(u_t)+\textup{AM}(v_t)-2\textup{AM}(P(u_t,v_t)).$
\end{proof}

The geodesic convexity and closedness of $d_1$-balls inside $\mathcal E^2$ is an immediate consequence:

\begin{corollary} For any $\rho > 0$ and $u \in \mathcal E^2(X,\o)$, the set 
$$B_\rho(u)=\{ v \in \mathcal E^2(X,\o), \ d_1(v,u) \leq \rho\}$$ 
is $d_2$--closed and $d_2$--convex, i.e., for any $v_0,v_1 \in B_\rho(u)$ the finite energy geodesic $[0,1] \ni t \to v_t \in \mathcal E^2$ connecting $v_0,v_1$ is contained in $B_\rho(u)$. 
\end{corollary}
\begin{proof} $d_2$--closedness follows from the fact that $d_2$ dominates $d_1$. Let $[0,1]\ni t \to v_t \in \mathcal E^2$ is a finite energy geodesic with $v_0,v_1\in B_{\rho}(u)$. By definition, since $\mathcal E^2 \subset \mathcal E^1$, the curve $t \to v_t$ is a finite energy geodesic inside $\mathcal E^1$ as well. By the previous proposition $t \to d_1(u,v_t)$ is convex, hence  $d_1(u,v_t) \leq \rho$. 
\end{proof}

The main result of this subsection is the following: 

\begin{theorem} \label{d1weakd2dominatethm} Suppose $\{u_k\}_k \subset \mathcal E^2$  is $d_2$-bounded and $u \in \mathcal E^2$. Then $d_1(u_k,u) \to 0$ if and only if $\| u_j - u\|_{L^1(X)} \to 0$ and  $u_k$ converges to $u$ $d_2$-weakly.  
\end{theorem}

\begin{proof} Assume first that $d_1(u_k,u) \to 0$. From \cite[Theorem 5(ii)]{da2} it follows that $\| u_j - u\|_{L^1(X)} \to 0$. As recalled in Proposition \ref{WeakdConvProp}, any subsequence of $\{ u_k\}_k$ contains a $d_2$--weakly convergent subsubsequence $u_{k_l}$, converging $d_2$-weakly to some $v \in \mathcal E^2$. We show that $v=u$. Indeed, for any $j \in \Bbb N$ the set  $B_{\frac{1}{j}}(u)$ is $d_2$--closed and $d_2$--convex by the previous corollary and for high enough $k_l$ we have $u_{k_l} \in B_{\frac{1}{j}}(u)$. As recalled in  Proposition \ref{WeakdConvProp}, it follows now that $v \in B_{\frac{1}{j}}(u)$ for all $j$, hence $v=u$. 

For the reverse direction, as $d_2$-boundedness gives that  $\AM(u_j)$ is uniformly bounded, by \cite[Proposition 5.9]{da2} it suffices to show that any convergent subsequence of $\AM(u_j)$ converges to $\AM(u)$. Assume that $u_{j_k}$ is such a subsequence and set $c=\lim_k \AM(u_{j_k})$. By definition, $u_{j_k}$ still converges $d_2$-weakly to $u$. For each $\vep>0$, consider the set 
$$
E_{\vep}:=\{\phi \in \Ec^2 :  c - \vep \leq \AM(\phi) \leq c + \vep\}.
$$
Since $d_2$ dominates $d_1$ and $\AM$ is $d_1$-continuous and linear along finite energy geodesics, it follows that $E_{\vep}$ is $d_2$-closed and $d_2$-convex. By Proposition \ref{WeakdConvProp}, it follows that $u\in E_{\vep}$. Letting $\vep\to 0$ we get $\AM(u)= c$, finishing the proof. 
\end{proof}

\begin{remark}\label{rem: weakd2convergence}
Using (the proof of) this last result it is possible to construct a $d_2$-bounded sequence $u_j \in \mathcal E^2$ converging $d_2$-weakly to some $u\in \Ec^2$, but $\Vert u_j-u\Vert_{L^1(X)} \not \to 0$. Indeed, one can construct a $d_2$-bounded sequence $u_j \in \mathcal E^2$ such that $\|u_j-v\|_{L^1(X)} \to 0$ for some $v \in \mathcal E^2$ but $\o^n_{u_j}$ does not converge weakly to $\o^n_v$, in particular $\AM(u_j)$ can not converge to $\AM(v)$. By Proposition \ref{WeakdConvProp} we can extract a subsequence, again denoted by $u_j$, such that  $u_j$ converges $d_2$-weakly to some $u \in \mathcal E^2$. By the last step in the proof of the previous theorem $\AM$ is weak $d_2$-continuous, hence $\AM(u_j) \to \AM(u)$, but we cannot have $u=v$ as $\AM(u) \neq \AM(v)$.
\end{remark}

\section{The weak twisted Calabi flow}
\label{sec: weak twisted Calabi}
As shown in \cite{da1}, the metric completion $(\mathcal E^2,d_2) = \overline{(\mathcal{H},d_2)}$ is a CAT(0) space. Suppose $\chi$ satisfies \eqref{eq: ChiProp}. By Theorem \ref{ExtKEnergyThm}, the extended $\mathcal K_\chi$ is $d_2$-lsc and convex on $\mathcal E^2$. By Theorem \ref{MayerThm} and Remark \ref{remark: initial potential}, the weak gradient flow $t \to c_{t}$ of $\mathcal K_\chi$ emanating from any $c_{0}\in  \mathcal{E}^2$ is well-defined and uniquely determined by the evolution variational inequality (\ref{eq:evi}).

When $\chi$ is smooth, the \emph{smooth twisted Calabi flow} is just a simple generalization of the usual smooth Calabi flow:
\begin{equation*}
\frac{d}{dt} c_t = S_{\o_{c_t}} - \bar S_{\chi} - \textup{Tr}^{\o_{c_t}}\chi.
\end{equation*}
\paragraph{Comparison with Streets' setting.}
In \cite{st1} another (a priori different) extension $\overline{\Kc}$
of the Mabuchi functional $\mathcal{M}$ on $\mathcal{H}$ to the
completion $\overline{(\mathcal{H},d_2)}=(\Ec^2,d_2)$ was considered, defined by
\[
\overline{\Kc}(\bar{u}):=\liminf_{d(u_{j},\bar{u}) \to 0}\mathcal{K}(u_{j})
\]
 where the infimum is taken over all sequences $u_{j}$ in $\mathcal{H}$
converging to $\bar{u}$ in $\overline{(\mathcal{H},d_2)}$. It is shown in \cite{st1}, that the functional $\overline{\mathcal{K}}$ thus
defined is $d_2$-lsc on $\overline{(\mathcal{H},d_2)}$, and then the author proceeds to study the gradient flow of $\overline{\mathcal K}$, dubbed the \emph{minimizing movement Calabi flow}.
By Theorem \ref{ExtKEnergyThm} we actually have $\overline{\mathcal{K}}=\Kc$, thus our finite energy Calabi flow coincides with the minimizing movement Calabi flow considered in \cite{st1}. One of the advantages of our consideration is that computations in $\Ec^2$ are explicit and avoid the difficulties of using Cauchy sequences.

\medskip
We show that the weak version of the twisted Calabi flow agrees with the smooth version as long as the latter exists. 
The following result was proved by J. Streets in the case $\chi=0$ using different methods.
\begin{proposition}
\label{prop:consistence}Suppose $\chi \geq 0$ is a smooth closed (1,1)-form. Given any initial point $c_{0}\in\mathcal{H}_\o$, the corresponding weak twisted Calabi flow $t \to c_t$ coincides with the smooth twisted Calabi flow, as long as the latter exists. 
\end{proposition}

\begin{proof}
By the uniqueness property in \cite[Theorem 4.0.4]{ags} for curves
$t \to c_{t}$ satisfying the \emph{evolution variational inequality \eqref{eq:evi}}
(which is shown by differentiating $d(c^1_{t},c^2_{t})$ for two different solutions $t\to c^1_{t}$ and $t \to c^2_{t})$ it is enough to show that a solution
$t \to h_{t}$ to the ordinary twisted Calabi flow with starting point $h_0 = c_0$ satisfies the inequality {\eqref{eq:evi}.}

Suppose $v\in\mathcal{H}_\o$ is arbitrary, fix a time $t=t_{0}$ and let $[0,1] \ni s \to u_s \in \mathcal H^\Delta_\o$ be the weak geodesic connecting $u_0=h_{t_0}$ and $u_1=v$. From \cite[Lemma 3.5]{bb} we get the following ``slope inequality'':
\begin{equation*}
\mathcal{K_\chi}(v)-\mathcal{K_\chi}(h_{t_0})\geq\int_{X}(\bar{S}-S_{\omega_{h_{t_0}}}+\textup{Tr}^{\o_{h_{t_0}}}\chi)\frac{du_{s}}{ds}_{|s=0}\omega_{h_{t_0}}^{n}.
\end{equation*}
Now, by the definition of the twisted Calabi flow the r.h.s above may be written
as minus the scalar product $\int_{X}\frac{dh_{t}}{dt}_{|t=t_{0}}\frac{du_{s}}{ds}_{|s=0}\omega_{h_{t_{0}}}^{n}.$
Since $v\in\mathcal{H}_\o$, the latter scalar product coincides with the
derivative at $t=t_{0}$ of the function $t\to d^{2}_2(h_{t},v)/2$
(by  \cite[Theorem 6]{c}, or rather by a formula appearing in
the proof of the latter theorem). This concludes the proof in the
case when $v\in\mathcal{H}_\o.$ 

We handle the general case: suppose $v \in \mathcal E^2$ and $\Kcc(v)<+\infty$.  Notice that it is enough to show the following 'integral' version of \eqref{eq:evi} (with $G=\Kcc$):
\begin{equation}\label{eviintegrated}
\frac{1}{2} (d^{2}_2(c_{t_1},v)-d^{2}_2(c_{t_0},v))\leq  (t_1-t_0)\Kcc(v)-\int^{t_1}_{t_0}\Kcc(c_t)dt,
\end{equation}
for any $t_0,t_1 \in [0,\infty)$, $t_0 \leq t_1$. Indeed, the l.h.s. is locally Lipschitz, whereas $t \to \Kcc(c_t)$ is smooth, hence we may divide both sides by $t_1 - t_0$ and take the limit $t_1 \to t_0$ to obtain \eqref{eq:evi}. By Theorem \ref{thm: Ep approximation with entropy} there exists a sequence $v_j \in \mathcal H_\o$  that $d_2$-converges to $v$ such that $\Kcc(v_j)$ converges to $\Kcc(v)$.  After integrating, by the first part of the proof estimate \eqref{eviintegrated} holds for $v_j$ in place of $v$. Letting $j \to \infty$, we obtain \eqref{eviintegrated} for $v$ as well.
\end{proof}

\begin{lemma}
\label{lem:energi func is constant along cal fl}The functional $\AM$
is constant along any weak twisted Calabi flow trajectory $t \to c_t$.
\end{lemma}
\begin{proof}
For a \emph{smooth} Calabi flow this follows directly from differentiating
along the flow, but here we have to proceed in a different manner. We can assume that $\AM(c_0)=0$, as  $\Kcc$ is invariant under adding constants. On the other hand for any $u,v\in \Ec^2$, $d_2(u-\AM(u),v-\AM(v))\leq d_2(u,v)$. Thus the variational construction of the weak Calabi flow (see Section \ref{subsect: gradient flow}) gives 'minimizing movement' $c_t^m$ with $\AM(c_t^m)=0, \forall m$. Since $\AM$ is continuous with respect to $d_2$ it follows that $\AM(c_t)=0$ for all $t$.
\end{proof}

Now we arrive at the main result of this section:

\begin{theorem}\label{thm: large time Calabi} Suppose $(X,\o)$ is a compact connected K\"ahler manifold and $\chi=\beta + i\ddbar f$ satisfies \eqref{eq: ChiProp}. The following statements are equivalent:
\vspace{-0.1in}
\begin{itemize}
\setlength{\itemsep}{1pt}
    \setlength{\parskip}{1pt}
    \setlength{\parsep}{1pt}  
\item[(i)] $\mathcal M^2_\chi \neq \emptyset$.
\item[(ii)] For any weak twisted Calabi flow trajectory $t \to c_t$ there exists $c_\infty \in \mathcal M^2_\chi$ such that $d_1(c_t,c_\infty) \to 0$ and $\textup{Ent}(e^{-f}\o^n,\o_{c_t}^n) \to \textup{Ent}(e^{-f}\o^n,\o_{c_\infty}^n)$.
\item[(iii)] Any weak twisted Calabi flow trajectory $t \to c_t$ is $d_2$-bounded.
\item[(iv)] There exists a weak twisted Calabi flow trajectory $t \to c_t$ and $t_j \to \infty$ for which the sequence $\{ c_{t_j}\}_j$ is $d_2$-bounded.
\end{itemize}
\end{theorem}

\begin{proof} We start with the direction (i)$\to$(ii). Let $t \to c_t$ be a weak twisted Calabi flow trajectory. Let $v \in \mathcal M^2_\chi$. From \eqref{eq:evi}  it follows that $d_2(v,c_t) \leq d_2(v,c_0)$, hence $t \to c_t$ is a $d_2$-bounded curve.

As observed in \cite{st2} Theorem \ref{BacakThm} guarantees the existence of $c_{\infty} \in \mathcal M^2_\chi$ such that $c_t \to c_\infty$ $d_2$-weakly. But $\{c_t\}_t$ is bounded in the $d_2$ metric and also $\mathcal K_\chi(c_t)$ is bounded. By Corollary \ref{KChiCompactnessCor} it follows that $\{c_t\}_t$ is $d_1$--relatively compact, i.e. each subsequence has a $d_1$--convergent subsubsequence. 
By Theorem \ref{d1weakd2dominatethm} we must have $d_1(c_t,c_\infty) \to 0$. 

In the definition of $\mathcal K_\chi$ all terms are  $d_1$-continuous except for the entropy term. Since $c_\infty$ is a minimizer, lower semi-continuity gives $\lim_{t\to \infty} \mathcal K_\chi(c_t)= \mathcal K_\chi(c_\infty)$. All this additionally implies $\textup{Ent}(e^{-f}\o^n,\o_{c_t}^n) \to \textup{Ent}(e^{-f}\o^n,\o_{c_\infty}^n)$.

The directions (ii)$\to$(iii)$\to$(iv) are trivial. We finish the proof by arguing that (iv)$\to$(i). Let $t \to c_t$ be a weak twisted Calabi flow trajectory and $\{c_{t_j}\}_j$ be a $d_2$-bounded sequence with $t_j \to \infty$. From Proposition \ref{d2weaklimits} and Corollary \ref{KChiCompactnessCor} it follows that there exists $c_\infty \in \mathcal E^2$ such that $d_1(c_{t_j},c_\infty) \to 0$ and by the lower semi-continuity of $\mathcal K_\chi$, we get that in fact $c_\infty \in \mathcal M^2_\chi$.
\end{proof}
In Theorem \ref{thm: large time Calabi}(ii) one would like to have convergence with respect to $d_2$. The next result confirms this in the case when the flow is bounded from below by some potential: 
\begin{proposition}
Suppose $(X,\omega)$ is a compact connected K\"ahler manifold and $\chi$ satisfies (\ref{eq: ChiProp}). Let $t\to c_t$ be a weak twisted  Calabi flow trajectory. If there exists $\psi\in \Ec^2$ such that $c_t\geq \psi, \forall t$ then $c_t$ converges in $d_2$ to a minimizer of $\Kcc$.	
\end{proposition}
\begin{proof}
	By Theorem \ref{thm: large time Calabi} we know that $t \to c_t$ converges in $d_1$ to some $u\in \Ec^2$, a minimizer of $\Kcc$. As $c_t \geq \psi$, by dominated convergence theorem and Theorem \ref{thm: comparison d2 and I2} we only have to prove that $ \int_X (c_t-c)^2\omega_{c_t}^n \rightarrow 0.$ For a fixed $s>0$ we have 
$$ \int_{\{|c_t-c|\leq s\}} (c_t-c)^2 \omega_{c_t}^n \leq s \int_X |c_t-c| \omega_{c_t}^n \rightarrow 0, $$
as $t\to \infty$, since $d_1(c_t,c)\to 0$. Thus it suffices to show that
\begin{equation} \label{eq: convergent imporovement}
\sup_{t>0} \int_{\{|c_t-c|>s\}} (c_t-c)^2 \omega_{c_t}^n \rightarrow 0
\end{equation}
as $s\to +\infty$. Since $d^2(c,c_t)$ is bounded, by Theorem \ref{thm: comparison d2 and I2} one can find a positive constant $C_1$ such that $\sup_X c_t\leq C_1$ for all $t>0$. By the comparison principle in $\Ec$ (see \cite{gz1}) one has 
	$$
	\int_{\{c_t-c>s\}}\omega_{c_t}^n\leq \int_{\{c_t-c>s\}} \omega_c^n \leq \int_{\{c<C_1-s\}} \omega_c^n,
	$$
	which yields
	\begin{equation*}
		\int_s^{\infty} \omega_{c_t}^n(c_t-c>r)rdr \leq \int_s^{\infty} \omega_c^n(c<C_1-r) rdr.
	\end{equation*}
	The right-hand side converges to $0$ as $s\to +\infty$ because $c\in \Ec^2$. Therefore, to prove (\ref{eq: convergent imporovement}) it remains to show that
	\begin{equation}
		\label{eq: convergent imporovement 2}
		\sup_{t>0} \int_s^{\infty} \omega_{c_t}^n (c_t-c<-r) rdr  \rightarrow 0.
	\end{equation}
	Since $\sup_X c_t$ is bounded from above and $c_t\geq \psi$, we can find $C_2>0$ such that 
	$$
	\{c_t-c<-r\}\subset \{\psi \leq C_2+(c_t-r)/2\}.
	$$
	Using $\omega_{c_t}^n \leq 2^n \omega_{c_t/2}^n$ and the comparison principle we arrive at
	\begin{eqnarray*}
	\int_s^{\infty} \omega_{c_t}^n(c_t-c<-r)rdr &\leq &  \int_s^{\infty} \omega_{c_t}^n(\psi<C_2+ (c_t-r)/2)rdr\\
	&\leq &  \int_s^{\infty} \omega_{\psi}^n(\psi<C_3-r/2)rdr,
	\end{eqnarray*}
	where $C_3=C_2+ C_1/2$.
	The last term converges to $0$ as $s\to +\infty$ because $\psi\in \Ec^2$. This proves (\ref{eq: convergent imporovement 2}) and completes the proof.
\end{proof}

Finally, we prove a result about geodesic rays weakly asymptotic to diverging weak Calabi flow trajectories. 

\begin{theorem}\label{thm: analog of Donaldson conjecture}  Suppose $(X,\o)$ is a compact connected K\"ahler manifold, $\chi\geq 0$ is smooth and Conjecture \ref{DRConj} holds. Let $[0,\infty) \ni t \to c_t \in \mathcal E^2$ be a weak twisted Calabi flow trajectory. Exactly one of the following holds: 
\vspace{-0.05in}
\begin{itemize}
\setlength{\itemsep}{1pt}
    \setlength{\parskip}{1pt}
    \setlength{\parsep}{1pt}  
\item[(i)] The curve $t \to c_t$ $d_1$-converges to a smooth twisted csc-K potential $c_\infty$.
\item[(ii)] $d_1(c_0,c_t) \to \infty$ as $t \to \infty$ and the curve $t \to c_t$ is $d_1$-weakly asymptotic to a finite energy geodesic $[0,\infty) \ni t \to u_t \in \mathcal E^1$ along which $\mathcal K_\chi$ decreases.
\end{itemize}
\vspace{-0.05in}
If  $\chi>0$, then independently of Conjecture \ref{DRConj} exactly one of the following holds:
\vspace{-0.05in}
\begin{itemize}
\setlength{\itemsep}{1pt}
    \setlength{\parskip}{1pt}
    \setlength{\parsep}{1pt}  
\item[(i')] The curve $t \to c_t$ $d_1$-converges to a unique minimizer in $\mathcal E^1$ of $\mathcal K_\chi$.
\item[(ii')] $d_1(c_0,c_t) \to \infty$ as $t \to \infty$ and the curve $t \to c_t$ is $d_1$-weakly asymptotic to a finite energy geodesic $[0,\infty) \ni t \to u_t \in \mathcal E^1$ along which $\mathcal K_\chi$ striclty decreases.
\end{itemize}
\end{theorem}

\begin{proof} Suppose (i) holds. Then $t \to c_t$ is $d_2$-bounded hence also $d_1$-bounded, hence it is impossible for (ii) to hold.

Now suppose (i) does not hold. By Corollary \ref{KChiCompactnessCor}  we must have $d_1(c_0,c_t) \to \infty$, otherwise there would exist $c_\infty \in \mathcal M^1_\chi$ smooth twisted csc-K, in particular $c_\infty \in \mathcal M^2_\chi$. By Theorem \ref{WeakCalFlowConv} this would imply that $(i)$ holds, a contradiction.

Let $[0,d_1(c_0,c_t)] \ni l \to u^t_l \in \mathcal E^1$ be the $d_1$--unit finite energy geodesic connecting $c_0,c_t$. By convexity of $l \to \mathcal K_\chi(u^t_l)$ it follows that 
$$\frac{\mathcal K_\chi(u^t_l) - \mathcal K_\chi(c_0)}{l}=\frac{\mathcal K_\chi(u^t_l) - \mathcal K_\chi(u^t_0)}{l} \leq \frac{\mathcal K_\chi (u^t_{d_t}) - \mathcal K_\chi (u^t_0)}{d_1(c_0,c_t)} =\frac{\mathcal K_\chi (c_t) - \mathcal K_\chi(c_0)}{d_1(c_0,c_t)}\leq 0,$$
hence $\{\mathcal K_\chi(u^t_l)\}_{t \in [0,\infty)}$ is uniformly bounded. As $d_1(u^t_l,u^t_0)=l$, we can apply Corollary \ref{KChiCompactnessCor} to give a $d_1$--converging subsequence to some $u_l \in \mathcal E^1$. Using a Cantor process, we can arrange for a subsequence $t_k$ such that for all $l \in \Bbb Q$ there exists $u_l \in \mathcal E^1$ such that
$d_1(u^{t_k}_l,u_l)\to 0$ as $k\to +\infty$ for each $l$. As we are dealing with the limit of $d_1$--unit speed geodesic segments, we will clearly have
$$d_1(u_{l_1},u_{l_2})=|l_1-l_2|, \ l_1,l_2 \in \Bbb Q_+.$$
Using equicontinuity, in the complete metric space $\mathcal E^1$ we can extend the curve $\Bbb Q_+ \ni l \to u_l \in \mathcal E^1$ to $d_1$-geodesic ray $[0,\infty) \ni l \to u_l \in \mathcal E^1$, satisfying $d_1(u^{t_k}_l,u_l) \to 0$ for all $l \in [0,\infty)$.

Using Proposition \ref{FinEnGeodStabProp} we additionally  obtain that $l \to u_l$ is in fact a finite energy geodesic. Because all functions $l \to \mathcal K_\chi(u^{t_k}_l)$  are uniformly bounded above and $\mathcal K_\chi$ is $d_1$-lsc, it necessarily follows that $l \to \mathcal K_\chi(u_l)$ is also bounded above. Convexity and boundedness now give that $l \to \mathcal K_\chi(u_l)$ is actually decreasing. 

Lastly, we focus on the case when $\chi >0$ is K\"ahler. In Theorem \ref{thm: uniqueness of twisted minimizer} we have proved that a minimizer of the twisted Mabuchi functional is unique if exists. Also, when (i') holds then by Remark \ref{remark: DRremark} the curve $t \to c_t$ is $d_1$-bounded, hence it is impossible for (ii') to hold. 

We assume that (i') does not hold. Let $t \to c_t$ be a weak twisted Calabi flow trajectory. We can assume that $c_t$ is $d_1$-divergent, otherwise Theorem \ref{EntropyCompactnessThm}
 would imply existence of a minimizer in $\mathcal E^1$. By the same argument as above, we can construct  a  weakly asymptotic finite energy geodesic ray $t\to u_t$ along which $\Kcc$ is decreasing. We claim that in fact $\Kcc$ is strictly decreasing along $t  \to u_t$. Indeed, if it were not the case, by convexity of $t \to \Kcc(u_t)$,  we would obtain that $t \to \Kcc(u_t)$ is constant  for $t$ greater than some $t_0>0$. By Lemma \ref{lem:energi func is constant along cal fl} $\AM$ is constant along $t \to c_t$, hence also along $t \to u_t$. As both $t \to \mathcal K(u_t),\AM_{\chi}(u_t)$ are convex, we obtain that $t \to \AM_{\chi}(u_t)$ is in fact linear and Theorem \ref{thm: linear along geodesic} then reveals that $u_t$ is stationary after $t_0$, contradicting the $d_1$-divergence of the ray $t \to u_t$.	
\end{proof}

\let\OLDthebibliography\thebibliography 
\renewcommand\thebibliography[1]{
  \OLDthebibliography{#1}
  \setlength{\parskip}{1pt}
  \setlength{\itemsep}{1pt plus 0.3ex}
}

\bigskip

{\sc Chalmers University of Technology}

{\tt robertb@chalmers.se} \\

{\sc University of Maryland}

{\tt tdarvas@math.umd.edu}\\

{\sc Scuola Normale Superiore, Pisa, Italy}

{\tt chinh.lu@sns.it}
\end{document}